\documentclass[12pt]{amsart}
\usepackage{amssymb, verbatim}
\usepackage{epsfig,color}
\usepackage{enumerate}
\usepackage{color,url}

\date{\today}

\newtheorem{theorem}{Theorem}[section]

\newtheorem{lemma}[theorem]{Lemma}

\theoremstyle{definition}

\theoremstyle{remark}

\makeatletter
\@namedef{subjclassname@2020}{%
 \textup{2020} Mathematics Subject Classification}
\makeatother

\numberwithin{equation}{section}

\usepackage[T1]{fontenc}

\newcommand{\hyper}[5]{\,{}_{#1}F_{#2}\left(\!\!%
\begin{array}{cc}{\displaystyle{#3}}\\[-0.1ex]
{\displaystyle{#4}} \end{array}\Big| \,{\displaystyle{#5}}
\right)}

\begin{document}

\title[Fourier transforms of some special functions]{Fourier transforms of orthogonal structures on the paraboloid}

\author[H. \"{O}zkan \c{C}etin]{Hasan \"{O}zkan \c{C}etin}
\address[H. \"{O}zkan \c{C}etin]{Ankara University, Faculty of Science, Department of Mathematics, 06100, Tando\u{g}an, Ankara, Türkiye}
\email[H. \"{O}zkan \c{C}etin]{hasanozkan.cetin@msu.edu.tr}

\author[R. Akta\c{s} Karaman]{Rab\.{I}a  Akta\c{s} Karaman}
\address[R. Akta\c{s} Karaman]{Ankara University, Faculty of Science, Department of Mathematics, 06100, Tando\u{g}an, Ankara, T\"{u}rkiye}
\email[R. Akta\c{s} Karaman (corresponding author)]{raktas@science.ankara.edu.tr}



\begin{abstract}
The purpose of this paper is to obtain Fourier transforms of multivariate orthogonal structures on the paraboloid such as Laguerre polynomials on the paraboloid and Jacobi polynomials on the paraboloid, and to define two new families of multivariate orthogonal functions by using Parseval's identity. In addition, some contiguous relations for these families of functions are given, and the obtained results are expressed in terms of the continuous Hahn polynomials.
\end{abstract}

\maketitle

\section{Introduction}

The theory of orthogonal polynomials originates from classical analysis and was significantly shaped by 19th-century mathematicians like Laplace, Legendre, Jacobi, and Chebyshev. A well-known starting point is the Legendre polynomial sequence, typically introduced as orthogonal over the interval
$[-1,1]$ with a uniform weight function. More broadly, the choice of weight function is central to establishing orthogonality, giving rise to various classical polynomial families such as Chebyshev, Hermite, Laguerre, and Jacobi-each designed for particular analytical contexts and applications. These one-variable families have been generalized in multiple directions: one through the adoption of nonstandard weights (including discrete and $q$-analogue forms), and another by transitioning into multivariate frameworks, where the geometry of the domain plays a decisive role.

\medskip

In this study, we focus on multivariate extensions of orthogonal polynomials together with integral transforms. These transforms allow functions to be moved between different domains, making it easier to analyze and solve complex problems. By combining a function with a suitable kernel, such methods can reveal hidden structures. For example, the Fourier transform connects the time domain with the frequency domain. Classical transforms such as Fourier, Laplace, Beta, Hankel, Mellin, and Whittaker play an important role both in theory and in applications, often appearing in the study of special functions (\cite{3}, \cite{22}-\cite{20}).

\medskip

Consider the Hermite functions, which are products of the Hermite polynomials $H_{n}(x)$ and the Gaussian exponential $\exp({-x^2}/2)$; they are known to be eigenfunctions of the Fourier transform \cite{7,9,19}. Previous works have shown, through the Fourier-Jacobi transform, that Jacobi polynomials can be transformed into Wilson polynomials, with related results linking them to continuous Hahn polynomials \cite{7,9}. Additional studies have explored the Fourier transforms of finite orthogonal polynomial families along with generalizations such as ultraspherical and symmetric orthogonal sequences (\cite{8}, \cite{11}-\cite{20}). In terms of extending orthogonal polynomials to several variables, important advances were made by Tratnik \cite{Trat1,Trat2}, who formulated multivariate versions of both continuous and discrete families from the Askey scheme, including their weight functions, hypergeometric representations, and biorthogonality properties. Koelink et al. \cite{Koe} investigated the convolution identities of Racah coefficients involving continuous Hahn, Hahn, and Jacobi polynomials. Recent studies have also investigated multivariable Fourier transforms and their role in developing new classes of orthogonal functions, particularly on domains such as the unit ball, simplex, and cone, where Parseval's identity plays a key role in establishing orthogonality relations (\cite{AA}, \cite{22}-\cite{24}).

\medskip

Special attention is given to orthogonal systems defined over parabolic regions. Xu \cite{Xu-paraboloid}, for instance, analyzed polynomial systems on paraboloid in $\mathbb{R}^{d+1}$ such as Jacobi polynomials on the paraboloid, demonstrating that these polynomials are eigenfunctions of a second-order differential operator, with eigenvalues depending both on the degree of the polynomial and on another index arising from the specific choice of the orthogonal basis.

\medskip
This work is inspired by recent research \cite{233} on the Fourier transform of orthogonal polynomials on the unit ball. In \cite{AA}, Fourier transforms of the Laguerre polynomials and Jacobi polynomials on the cone have been studied. Our main goal is to investigate the Fourier transforms of multivariate orthogonal polynomials on the paraboloid by using the same approach applied on the cone. Our focus lies on the Jacobi and Laguerre polynomial families in this setting. By applying Fourier transform techniques together with Parseval's identity, we construct two families of special functions, which are expressed through continuous Hahn polynomials.

\medskip

The present study is structured into four main sections. Section 2 is devoted to preliminary concepts and serves to recall essential results concerning classical orthogonal polynomials defined on both the unit ball and the parabolic domain. In Section 3, we explain the main contributions of this work, comprising six central results. Specifically, we provide explicit expressions for the Fourier transforms of Laguerre and Jacobi polynomials on the paraboloid, as well as the derivation of associated families of special functions through the application of Fourier analysis and Parseval's identity. These results are formulated through continuous Hahn polynomials and some contiguous relations for these families of functions are presented. The last section presents the proofs of the aforementioned results, thereby substantiating the theoretical framework established in the preceding sections.

\section{Preliminary Results}\label{sec:bdn}
Here, we outline the necessary background on multivariate orthogonal polynomials, focusing on fundamental properties of classical multivariate systems on both the unit ball and the paraboloid.\\

Throughout this paper, we will use multi-index notation (see \cite{29}):  ${\mathbf{k}}=(k_{1},\ldots,k_{d})\in{\mathbb{N}}_{0}^{d}$ and ${\boldsymbol{x}}=(x_{1},\ldots,x_{d})\in{\mathbb{R}}^{d}$. Assume that $w$ is a weight function on a domain $\Omega \subset{\mathbb{R}%
}^{d}$. Let $\Pi_{m}^{d}$ denote the space of polynomials of degree at most $m$ in $d$
variables. Let the inner product ${\langle}\cdot,\cdot{\rangle}_{w}$ on the space of these polynomials be defined by
\[
{\langle}P,Q{\rangle}_{w}=\int_{\Omega}P(\boldsymbol{x})Q(\boldsymbol{x})w(\boldsymbol{x})\mathrm{d}\boldsymbol{x}
\]
where $d\boldsymbol{x}=\mathrm{d}x_{1}\cdots \mathrm{d}x_{d}.$ if
\[
{\langle}P,Q{\rangle}_{w}=0,\qquad \forall Q\in \Pi_{m-1}^{d},
\]
the polynomial $P$ of degree $m$ is an orthogonal polynomial with respect to this inner product. Let ${\mathcal{U}}_{m}%
^{d}(\Omega,w)$ be the space of orthogonal polynomials of degree $m$ in $d$ variables with respect to this inner product. Then (see \cite{29})
\[
r_{m}^{d}:=\dim{\mathcal{U}}_{m}^{d}(\Omega,w)=\binom{m+d-1}%
{m},\qquad m=0,1,2,\ldots.
\]

 When $d>1$, the space ${\mathcal{U}}_{m}^{d}(\Omega,w)$  has infinitely many possible bases. Furthermore, since orthogonality is defined as orthogonal to all polynomials of lower degrees, the elements of a given basis are not necessarily mutually orthogonal. If
\[
\langle P^m_j, P^m_k \rangle_w = 0 \quad \text{for all } j\neq k,
\]
then a basis $\{P^m_j : 1 \leq j \leq r^d_m\}$ of ${\mathcal{U}}_{m}^{d}(\Omega,w)$  is called a mutually orthogonal basis.

\subsection{Orthogonal polynomials on the unit ball}
Let  $\left \Vert \boldsymbol{x}\right \Vert :=\left(  x_{1}^{2}+\cdots+x_{d}^{2}\right)  ^{1/2}$ for $\boldsymbol{x=}\left(  x_{1},\dots,x_{d}\right)  \in \mathbb{R}^{d}$. For $\mu>-\frac{1}{2}$, let $w_{\mu}$ be the weight function on the unit ball
${\mathbb{B}}^{d}=\{\boldsymbol{x}\in{\mathbb{R}}^{d}:\Vert \boldsymbol{x}\Vert \leq1\}$ (see \cite{29})
\[
w_{\mu}(\boldsymbol{x})=(1-\Vert \boldsymbol{x}\Vert^{2})^{\mu-\frac{1}{2}},\quad \mu
>-\tfrac{1}{2},\quad \boldsymbol{x}\in{\mathbb{B}}^{d}.
\]

In the case $d=1$, the associated orthogonal polynomials are the classical Gegenbauer polynomials $C_{m}^{\mu}$ defined by \cite[p. 277, Eq. (4)]{25}
\begin{equation}\label{hyper}
C_{m}^{\left(  \mu \right)  }\left(  x\right)  =\frac{\left(
2\mu\right)  _{m}}{m!} \,  \hyper{2}{1}{-m,m+2\mu}{\mu+\frac{1}{2}}{\frac{1-x}{2}},
\end{equation}
where $_{2}F_{1}$ denotes the Gauss hypergeometric function with the special case $p=2,q=1$ of the generalized hypergeometric function
(cf. \cite{26})
\begin{equation}
\hyper{p}{q}{\alpha_{1},\alpha_{2},\dots,\alpha_{p}}{\beta_{1},\ \beta_{2},\dots,\beta_{q}}{x}=
{\displaystyle \sum \limits_{m=0}^{\infty}}
\frac{\left(  \alpha_{1}\right)  _{m}\left(  \alpha_{2}\right)  _{m} \dots \left(
\alpha_{p}\right)  _{m}}{\left(  \beta_{1}\right)  _{m}\left(  \beta_{2}\right)
_{m}\dots \left(  \beta_{q}\right)  _{m}}\frac{x^{m}}{m!}\label{genhyper}%
\end{equation}
where $\left(\alpha \right)_{m}$ is called Pochhammer symbol and it is defined as $\left(\alpha \right)_{m}=\alpha \left(  \alpha+1\right) \cdots \left(\alpha+m-1\right),~m\geq1$, $\left(\alpha \right)_{0}=1$. The Gegenbauer polynomials are orthogonal with respect to the weight function $w(x)=\left(  1-x^{2}\right)^{\mu-\frac{1}{2}}$ over the interval $[-1,1]$. Indeed, it follows \cite[p.281, Eq. (28)]{25}
\begin{equation}
 \int \limits_{-1}^{1}
C_{m}^{\left(  \mu \right)  }\left(  x\right)  C_{n}^{\left(
\mu \right)  }\left(  x\right)\left(  1-x^{2}\right)  ^{\mu-\frac{1}{2}}dx=h_{m}^{\mu}~\delta_{m,n}%
,\label{ort} \quad
 \left(  m,n\in \mathbb{N}_{0}:= \mathbb{N}\cup \left \{0\right \}  \right),
\end{equation}
where $\delta_{m,n}$ is the Kronecker delta, $h_{m}^{\mu}$ denotes the norm square given as
\begin{equation}
h_{m}^{\mu}=\frac{\left(  2\mu \right)  _{m}\Gamma \left(  \mu
+\frac{1}{2}\right)  \Gamma \left(  \frac{1}{2}\right)  }{m!\left(
m+\mu \right)  \Gamma \left( \mu \right)  }\label{gnorm}
\end{equation}
and, the Gamma function $\Gamma \left(  x\right)$ is defined by (cf. \cite{26})%
\begin{equation}
\Gamma \left(  x\right)  =\int \limits_{0}^{\infty}s^{x-1}e^{-s}ds,\  \Re \left(
x\right)  >0.\label{2}%
\end{equation}
Let ${\mathcal{U}}_{m}^{d}({\mathbb{B}}^{d},w_{\mu})$ be the space of orthogonal polynomials of degree $m$ for the weight function $w_{\mu}(\boldsymbol{x})$.  The orthogonal polynomials of degree $m$ are eigenfunctions of a second-order differential operator \cite[p.141, Eq. (5.2.3)]{29}: for $P\in{\mathcal{U}}_{m}^{d}({\mathbb{B}}^{d},w_{\mu})$
\begin{equation*}
\left(  \Delta-\left \langle x,\nabla \right \rangle ^{2}-\left(  2\mu
+d-1\right)  \left \langle x,\nabla \right \rangle \right)  P=-m\left(
m+2\mu+d-1\right)P
\end{equation*}
where $\Delta$ and $\nabla$ denote the Laplace operator and the gradient operator, respectively. When $d>1$, this space possesses a variety of distinct bases; however, we focus primarily on those orthogonal polynomials that can be constructed in terms of Gegenbauer polynomials $C_{m_{j}}^{\left(\lambda_{j}\right)}\left(x_{j}\right)$ as follows \cite[p.143]{29}
\begin{equation}
P_{\mathbf{m}}^{\mu}\left(  \boldsymbol{x}\right)  =\prod \limits_{j=1}%
^{d}\left(  1-\left \Vert \boldsymbol{x}_{j-1}\right \Vert ^{2}\right)
^{\frac{m_{j}}{2}}C_{m_{j}}^{\left(  \lambda_{j}\right)  }\left(  \frac{x_{j}%
}{\sqrt{1-\left \Vert \boldsymbol{x}_{j-1}\right \Vert ^{2}}}\right) \label{P}%
\end{equation}
where $\lambda_{j}=\mu+\left \vert \mathbf{m}^{j+1}\right \vert +\frac{d-j}{2},$
and
\begin{equation}\label{notation}
\begin{cases}
\boldsymbol{x}_{0}   =0,\  \  \boldsymbol{x}_{j}=\left(  x_{1},\dots,x_{j}\right)  , \\
\mathbf{m}   =\left(  m_{1},\dots,m_{d}\right)  ,\  \text{\ }\left \vert
\mathbf{m}\right \vert =m_{1}+\cdots+m_{d}=m,\\
\mathbf{m}^{j}   =\left(  m_{j},\dots,m_{d}\right)  ,\text{ \ }\left \vert
\mathbf{m}^{j}\right \vert =m_{j}+\cdots+m_{d},\  \text{\  \ }1\leq j\leq d,
\end{cases}
\end{equation}
and $\mathbf{m}^{d+1}:=0.$ They satisfy the following orthogonality relation
\[
\int \limits_{\mathbb{B}^{d}}
P_{\mathbf{m}}^{\mu}\left(  \boldsymbol{x}\right)  P_{\mathbf{n}}^{\mu}\left(
\boldsymbol{x}\right)w_{\mu}\left(  \boldsymbol{x}\right)d\boldsymbol{x}=h_{\mathbf{m}}^{\mu}\delta_{\mathbf{m}%
,\mathbf{n}}
\]
where $\delta_{\mathbf{m},\mathbf{n}}=\delta_{m_{1},n_{1}}\cdots \delta
_{m_{d},n_{d}}$ and $h_{\mathbf{m}}^{\mu}$ denotes the norm square given as (cf. \cite{29})%
\begin{equation}
h_{\mathbf{m}}^{\mu}=\frac{\pi^{d/2}\Gamma \left(  \mu+\frac{1}{2}\right)
\left(  \mu+\frac{d}{2}\right)  _{\left \vert \mathbf{m}\right \vert }}%
{\Gamma \left(  \mu+\frac{d+1}{2}+\left \vert \mathbf{m}\right \vert \right)
}\prod \limits_{j=1}^{d}\frac{\left(  \mu+\frac{d-j}{2}\right)  _{\left \vert
\mathbf{m}^{j}\right \vert }\left(  2\mu+2\left \vert \mathbf{m}^{j+1}%
\right \vert +d-j\right)  _{m_{j}}}{m_{j}!\left(  \mu+\frac{d-j+1}{2}\right)
_{\left \vert \mathbf{m}^{j}\right \vert }}.\label{Norm}%
\end{equation}

\subsection{Orthogonal polynomials on the paraboloid}

We recall orthogonal structures on the solid paraboloid of the revolution
\begin{equation}\label{eq:conedomain}
{\mathbb{U}}^{d+1}=\{(t,\boldsymbol{x})\in{\mathbb{R}}^{d+1}:\, {\Vert \boldsymbol{x}\Vert}^2 \leq
t,\,\boldsymbol{x}\in{\mathbb{R}}^{d},\,0\leq t \leq b\},
\end{equation}
which is bounded by the surface ${{\mathbb{U}}_0}^{d+1}=\{(t,\boldsymbol{x}): {\Vert \boldsymbol{x}\Vert}=
\sqrt{t},\,\boldsymbol{x}\in{\mathbb{R}}^{d},\,0\leq t \leq b\}$ and the hyperplane $t = b$.
The orthogonality is given with respect to the weight function \cite{Xu-paraboloid}
\begin{equation}
W_{\mu}(t,\boldsymbol{x})=\rho(t)(t-\Vert \boldsymbol{x}\Vert^{2})^{\mu-\frac{1}{2}},\quad \mu
>-\tfrac{1}{2}\label{eq:Wmu-paraboloid}%
\end{equation}
where $\rho$ is the Jacobi weight function or the Laguerre weight function. In the case of $b=1$, Xu \cite{Xu-paraboloid} studied a family of orthogonal polynomials with respect to the weight function $\rho(t)=t^{{\beta}}(1-t)^{{\gamma}}$, which are called Jacobi polynomials in the paraboloid. It was shown that these polynomials satisfy a second-order differential equation, acting as eigenfunctions of the associated operator. However, unlike the cases of the cone and the hyperboloid where the eigenvalue structure remains invariant under different orthogonal bases, the corresponding eigenvalues here depend not only on the polynomial degree but also on an additional index determined by the specific choice of orthogonal basis. On the other hand, the case $b=\infty$ gives Laguerre polynomials on the paraboloid, with the weight function $\rho(t)=t^{{\beta}}{\mathrm{e}}^{-t}$.

We define the inner product
\[
{\langle}f,g{\rangle}_{\mu}:=\int_{{\mathbb{U}}^{d+1}}f(t,\boldsymbol{x})g(t,\boldsymbol{x})W_{\mu
}(t,\boldsymbol{x})\mathrm{d}\boldsymbol{x}\mathrm{d}t
\]
on the paraboloid ${\mathbb{U}}^{d+1}$. By using the change of variable, the following equality is satisfied
\[
\int_{{\mathbb{U}}^{d+1}}f(t,\boldsymbol{x})\mathrm{d}\boldsymbol{x}\mathrm{d}t=\int
_{0}^{b}t^{d/2}\int_{{\mathbb{B}}^{d}}f(t,\sqrt{t}\boldsymbol{y})\mathrm{d}\boldsymbol{y}\,\mathrm{d}t.
\]
For $m=0,1,2,...,$ let ${\mathcal{U}}_{m}({\mathbb{U}}^{d+1},W_{\mu})$ denote the space of orthogonal
polynomials of degree $m$ in $(t,\boldsymbol{x})$ with respect to the inner product
${\langle}.,.{\rangle}_{\mu}$ on the paraboloid.

A basis of this space can be given in terms of orthogonal polynomials on the unit ball and a family of polynomials in one variable such as Jacobi polynomials and Laguerre polynomials. These two families are given as follows.

\subsubsection{Jacobi polynomials on the paraboloid}

In the case of $b=1$ and $\rho(t)=t^{{\beta}}(1-t)^{{\gamma}}$, the Jacobi polynomials on the paraboloid are orthogonal with respect to the weight function $$W_{\beta,\gamma,\mu}(t,\boldsymbol{x})=t^{{\beta}}(1-t)^{{\gamma}}(t-\Vert \boldsymbol{x}\Vert^{2})^{\mu-\frac{1}{2}},$$ and are defined in terms of orthogonal polynomials on the unit ball and Jacobi polynomials in \cite{Xu-paraboloid} as follows.\\
Let $\beta>-\frac{d+1}{2}$ and $\gamma>-1$. Let ${\mathbb{P}}_{n} = \{P_{\mathbf{k}}^n: |{\mathbf{k}}| =n\}$ be an orthogonal basis with parity of ${\mathcal{U}}_{m}({\mathbb{B}}^{d},w_{\mu})$. For $n \le m$, Jacobi polynomials on the paraboloid are defined as
\begin{equation}
\label{eq:sQcone}{\mathsf{Q}}_{{\mathbf{k}},m}^{n} (t,\boldsymbol{x}) = P_{m-n}^{n+\beta+\mu+\frac{d-1}{2},\gamma}(1-2t) t^{n/2} P_{\mathbf{k}}^n \left( \frac{\boldsymbol{x}}{\sqrt{t}}\right) , \qquad|{\mathbf{k}}|
=n, \quad0 \le n \le m.
\end{equation}
Then ${\mathbb{Q}}_{n,m} = \{ {\mathsf{Q}%
}_{{\mathbf{k}},m}^{n}: |{\mathbf{k}}| =n, \, 0 \le n \le m\}$ is an orthogonal basis of
${\mathcal{U}}_{m}({\mathbb{U}}^{d+1}, W_{\beta,\gamma,\mu})$ \cite{Xu-paraboloid}.

Jacobi polynomial $P_{m}^{({\alpha},{\beta})}$ is defined, for ${\alpha},{\beta}>-1$, by
\begin{align}
P_{m}^{(\alpha,\beta)}(t)=\frac{({\alpha}+1)_{m}}{m!} \hyper{2}{1}{-m,m+{\alpha}+{\beta}+1}{{\alpha}+1}{\frac{1-t}{2}},
\end{align}\label{200}
in terms of the hypergeometric function and it satisfies the orthogonal relation \cite{25}
\begin{align}
\int_{-1}^{1}P_{m}^{({\alpha},{\beta})}(t)P_{n}^{({\alpha},{\beta}%
)}(t)(1-t)^{{\alpha}}(1+t)^{{\beta}}\mathrm{d}t=h_{m}^{({\alpha},{\beta}%
)}\delta_{m,n} \label{ort-J},
\end{align}
where
\[
h_{m}^{({\alpha},{\beta})}=\frac{2^{\alpha+\beta+1}\Gamma \left(  m+\alpha+1\right)  \Gamma \left(
m+\beta+1\right)  }{\left(  2m+\alpha+\beta+1\right)  \Gamma \left(
m+\alpha+\beta+1\right)  m!}.
\]

\subsubsection{Laguerre polynomials on the paraboloid}

In the case of $b=\infty$ and $\rho(t)=t^{{\beta}}{\mathrm{e}}^{-t}$, Laguerre polynomials on the paraboloid are orthogonal with respect to the weight function $$W_{\beta,\mu}(t,\boldsymbol{x})=t^{{\beta}}{\mathrm{e}}^{-t}(t-\Vert \boldsymbol{x}\Vert^{2})^{\mu-\frac{1}{2}},$$ and they are given in terms of orthogonal polynomials on the unit ball and Laguerre polynomials as follows.\\
Let $\beta>-\frac{d+1}{2}$ and ${\mathbb{P}}_{n} = \{P_{\mathbf{k}}^n: |{\mathbf{k}}| =n\}$ be an orthogonal basis with parity of ${\mathcal{U}}_{m}({\mathbb{B}}^{d},w_{\mu})$. For $n \le m$, Laguerre polynomials on the paraboloid are defined as
\begin{equation}
\label{eq:sQpara}{\mathsf{R}}_{{\mathbf{k}},m}^n (t,\boldsymbol{x}) = L_{m-n}^{n+\beta+\mu+\frac{d-1}{2}}(t) t^{n/2} P_{\mathbf{k}}^n \left( \frac{\boldsymbol{x}}{\sqrt{t}}\right) , \qquad|{\mathbf{k}}|
=n, \quad0 \le n \le m
\end{equation}
where
Laguerre polynomial $L_{m}^{{\alpha}}$ defined, for ${\alpha}>-1$, by
\begin{equation}
L_{m}^{{\alpha}}(t)=\frac{({\alpha}+1)_{m}}{m!} \hyper{1}{1}{-m}{\alpha+1}{t} =\frac{({\alpha}+1)_{m}}{m!}\sum_{k=0}^{m}\frac{(-m)_{k}}{({\alpha
}+1)_{k}k!}t^{k}\label{102}
\end{equation}
satisfies the orthogonal relation \cite{25}
\begin{align}
\int_{0}^{\infty}L_{m}^{{\alpha}}(t)L_{n}^{{\alpha}}(t)t^{{\alpha}}%
{\mathrm{e}}^{-t}\mathrm{d}t=\frac{\Gamma({\alpha}+m+1)}{m!}\delta_{m,n}\label{ort-L}.
\end{align}
Then ${\mathbb{R}}_{n,m} = \{ {\mathsf{R}%
}_{{\mathbf{k}},m}^n: |{\mathbf{k}}| =n, \, 0 \le n \le m\}$ is an orthogonal basis of
${\mathcal{U}}_{m}({\mathbb{U}}^{d+1}, W_{\beta,\mu}).$ Indeed, it follows from the definition of the polynomial
\begin{align*}
&  {\langle}\mathsf{R}_{{\mathbf{k}},m}^{n},\mathsf{R}_{{\mathbf{k}}^{\prime
},m^{\prime}}^{n^{\prime}}{\rangle}_{\beta,\mu}\\
\phantom{olaola} &  =\int_{0}^{\infty}t^{\frac{1}{2}\left(  d-1+n+n^{\prime
}\right)  +\mu+\beta}L_{m-n}^{\alpha_{n}}(t)L_{m^{\prime}-n^{\prime}}%
^{\alpha_{n^{\prime}}}(t)e^{-t}\mathrm{d}t\int_{{\mathbb{B}}^{d}%
}P_{{\mathbf{k}}}^{n}(\boldsymbol{y})P_{{\mathbf{k}}^{\prime}}^{n^{\prime}}\!(\boldsymbol{y})w
_{\mu}(\boldsymbol{y})\mathrm{d}\boldsymbol{y}\\
\phantom{olaola} &  =\int_{0}^{\infty}\!L_{m-n}^{\alpha_{n}}(t)L_{m^{\prime
}-n}^{\alpha_{n}}t^{\alpha_{n}}e^{-t}\mathrm{d}t\! \int_{{\mathbb{B}}^{d}%
}\!P_{{\mathbf{k}}}^{n}(\boldsymbol{y})P_{{\mathbf{k}}^{\prime}}^{n}(\boldsymbol{y})w_{\mu
}(\boldsymbol{y})\mathrm{d}\boldsymbol{y}\, \delta_{n,n^{\prime}}\\
\phantom{olaola} &  ={\Vert L_{m-n}^{{\alpha}_{n}}\Vert^{2}}\int_{{\mathbb{B}%
}^{d}}P_{{\mathbf{k}}}^{n}(\boldsymbol{y})P_{{\mathbf{k}}^{\prime}}^{n}(\boldsymbol{y})w
_{\mu}(\boldsymbol{y})\mathrm{d}\boldsymbol{y}\, \delta_{m,m^{\prime}}\delta_{n,n^{\prime}},
\end{align*}
it is seen from the orthogonality of polynomials on the unit ball where $\alpha_{n}=n+\mu+{\beta+}\frac{d-1}{2}.$

For the reader's convenience, we also provide the definitions of the continuous Hahn polynomials and the beta function, which will be applied in the subsequent sections. The continuous Hahn polynomials is expressed in terms of the $_{3}F_{2}$ hypergeometric function as \cite{27}
\begin{equation}
p_{m}\left( x;\alpha,\beta,\gamma,\delta\right)
=i^{m}\frac{\left( \alpha+\gamma\right) _{m}\left(\alpha+\delta\right) _{m}}{m!} \hyper{3}{2}{-m,\ m+\alpha+\beta+\gamma+\delta-1,\ \alpha+ix}{\alpha+\gamma,\ \alpha+\delta}{1} \label{hahn}
\end{equation}
and the beta function is defined as \cite{26}
\begin{equation}
B\left(\alpha,\beta\right)  ={\displaystyle \int \limits_{0}^{1}}x^{\alpha-1}\left(
 1-x\right)  ^{\beta-1}dx=\frac{\Gamma \left(\alpha\right)\Gamma \left( \beta\right)  }{\Gamma \left(\alpha+\beta\right)  }, \qquad \Re\left( \alpha\right)  ,\Re\left( \beta\right)  >0. \label{beta}
\end{equation}

\subsection{Fourier transform of orthogonal polynomials on the ball}

The univariate Fourier transform of a function $g(x)$, which is both absolutely integrable and square-integrable, is defined as follows \cite[p.111, Eq. (7.1)]{3}
\begin{equation}
\mathcal{F}
\left(  g\left(  x\right)  \right) (\xi)=\int \limits_{-\infty}^{\infty}e^{-i\xi
x}g\left(  x\right)  dx \label{16}%
\end{equation}
and in the multivariate case with $d$ variables, the Fourier transform of a function $g(x_{1},\dots,x_{d})$, both absolutely integrable and square-integrable, is defined as follows \cite[p. 182, Eq. (11.1a)]{3}
\begin{equation}
\mathcal{F}%
\left(  g\left(  x_{1},\dots,x_{d}\right)  \right)(\xi_{1},\dots,\xi_{d})  =\int \limits_{-\infty
}^{\infty}\cdots \int \limits_{-\infty}^{\infty}e^{-i\left(  \xi_{1}x_{1}%
+\cdots +\xi_{d}x_{d}\right)  }g\left(  x_{1},\dots,x_{d}\right)  dx_{1}\cdots dx_{d}.
\label{17}%
\end{equation}

The corresponding Parseval's identity associated with equation (\ref{16}) can be stated as follows
\cite[p.118, Eq. (7.17)]{3}
\begin{equation}
\int \limits_{-\infty}^{\infty}g\left(  x\right)  \overline {h\left(  x\right)
}dx=\frac{1}{2\pi}\int \limits_{-\infty}^{\infty}\mathcal{F}
\left(  g\left(  x\right)  \right)  \overline{ \mathcal{F}\left(  h\left(  x\right)  \right)  }d\xi,\label{monoparseval}%
\end{equation}
and for the multivariate case with $d$ variables, the corresponding Parseval's identity associated with equation (\ref{17}) takes the form \cite[p. 183, (iv)]{3}
\begin{multline}
 \int \limits_{-\infty}^{\infty}\cdots \int \limits_{-\infty}^{\infty}g\left(
x_{1},\dots,x_{d}\right)  \overline{h\left(  x_{1},\dots,x_{d}\right)  }
dx_{1}\cdots dx_{d} \\
 =\frac{1}{\left(  2\pi \right)  ^{d}}\int \limits_{-\infty}^{\infty}
\cdots \int \limits_{-\infty}^{\infty}\mathcal{F}\left(  g\left(  x_{1},\dots,x_{d}\right)  \right)  \overline{
\mathcal{F}\left(  h\left(  x_{1},\dots,x_{d}\right)  \right)  }d\xi_{1}\cdots d\xi_{d}.\label{multi}
\end{multline}

To proceed, we first summarize the established results concerning the Fourier transform of orthogonal polynomials on the unit ball, as reported in \cite{233}. Assume that the function is represented via an expansion of orthogonal polynomials on the unit ball, given by
\begin{multline}
g_{d}\left( \boldsymbol{x};\mathbf{k},\alpha,\mu \right)  :=g_{d}\left(x_{1},\dots,x_{d};k_{1},\dots,k_{d},\alpha,\mu \right)  \\
 =\prod \limits_{j=1}^{d}\left(  1-\tanh^{2}x_{j}\right)  ^{\alpha+\frac{d-j}{4}}\ P_{\mathbf{k}}^{\mu}\left(  \vartheta_{1},\dots,\vartheta_{d}\right),\label{150}%
\end{multline}
for $d\geq1$, where $\alpha,\mu$ are real parameters and
\begin{align}
\vartheta_{1}\left(  x_{1}\right)   & =\vartheta_{1}=\tanh x_{1},\nonumber\\
\vartheta_{2}\left(  x_{1},x_{2}\right)   & =\vartheta_{2}=\tanh x_{2}\sqrt{\left(  1-\tanh^{2}x_{1}\right)},\nonumber\\
\vartheta_{3}\left(  x_{1},x_{2},x_{3}\right)   & =\vartheta_{3}=\tanh x_{3}\sqrt{\left(  1-\tanh^{2}x_{1}\right)\left(  1-\tanh^{2}x_{2}\right)},\nonumber\\
\vdots\nonumber\\
\vartheta_{d}\left(  x_{1},\dots,x_{d}\right)   & =\vartheta_{d}=\tanh x_{d}%
\sqrt{\left(  1-\tanh^{2}x_{1}\right)  \left(  1-\tanh^{2}x_{2}\right)
\cdots \left(  1-\tanh^{2}x_{d-1}\right)  }\label{var1},
\end{align}
for $d\geq1$.

The following lemma, as established in \cite{233}, provides the Fourier transform of
$ g_{d}\left( \boldsymbol{x};\mathbf{k},\alpha,\mu \right)$, defined in (\ref{150}), obtained through a recursive procedure.
\begin{lemma}\cite{233}
\label{prop:OPcone2}
An explicit formula for the Fourier transform of the function $g_{d}\left( \boldsymbol{x};\mathbf{k},\alpha,\mu \right)  $ is given by
\begin{multline}%
\mathcal{F} \left(  g_{d}\left( \boldsymbol{x};\mathbf{k},\alpha,\mu \right)  \right)  :=
\mathcal{F}
\left(  g_{d}\left(  x_{1},\dots,x_{d};k_{1},\dots,k_{d},\alpha,\mu \right)  \right) \\
=2^{2d\alpha+\frac{d\left(  d-5\right)  }{4}+\text{$
{\textstyle \sum \limits_{j=1}^{d-1}}
$}{jk}_{j+1}}\prod \limits_{j=1}^{d}\left \{  \frac{\left(  2\left(
\left \vert \mathbf{k}^{j+1}\right \vert +\mu+\frac{d-j}{2}\right)  \right)
_{k_{j}}}{k_{j}!}\varphi_{j}^{d}\left(  \alpha,\mu,\mathbf{k};\xi_{j}\right)
\right \}  ,  \label{18}%
\end{multline}
where%
\begin{multline*}
\varphi_{j}^{d}\left(  \alpha,\mu,\mathbf{k};\xi_{j}\right)  =B\left(\alpha+\frac{\left \vert \mathbf{k}^{j+1}\right \vert +i\xi_{j}}{2}+\frac{d-j}{4},\alpha+\frac{\left \vert \mathbf{k}^{j+1}\right \vert -i\xi_{j}}{2}+\frac{d-j}{4}\right)    \\
\times
\hyper{3}{2}{-k_{j},\ k_{j}+2\left(  \left \vert \mathbf{k}^{j+1}\right \vert +\mu+\frac{d-j}{2}\right)  ,\ \alpha+\frac{\left \vert \mathbf{k}^{j+1}\right \vert +i\xi_{j}}{2}+\frac{d-j}{4}}{\left \vert \mathbf{k}^{j+1}\right \vert +\mu+\frac{d-j+1}{2},\  \left \vert \mathbf{k}^{j+1}\right \vert +2\alpha+\frac{d-j}{2}}{1},
\end{multline*}
or alternatively, expressed through the continuous Hahn polynomials
\begin{multline*}
\varphi_{j}^{d}\left(  \alpha,\mu,\mathbf{k};\xi_{j}\right)  =\frac{k_{j}!}{i^{k_{j}}\left(  \left \vert \mathbf{k}^{j+1}\right \vert +\mu+\frac{d-j+1}{2}\right)  _{k_{j}}\left(  \left \vert \mathbf{k}^{j+1}\right \vert +2\alpha+\frac{d-j}{2}\right)  _{k_{j}}}   \\
\times B\left(  \alpha+\frac{\left \vert \mathbf{k}^{j+1}\right \vert +i\xi_{j}}{2}+\frac{d-j}{4},\alpha+\frac{\left \vert \mathbf{k}^{j+1}\right \vert -i\xi_{j}}{2}+\frac{d-j}{4}\right)    \\
\times p_{k_{j}}\left(  \frac{\xi_{j}}{2};\alpha+\frac{\left \vert \mathbf{k}^{j+1}\right \vert }{2}+\frac{d-j}{4},\mu-\alpha+\frac{\left \vert \mathbf{k}^{j+1}\right \vert +1}{2}+\frac{d-j}{4}\right.    \\
\left.  ,\mu-\alpha+\frac{\left \vert \mathbf{k}^{j+1}\right \vert +1}{2}+\frac{d-j}{4},\alpha+\frac{\left \vert \mathbf{k}^{j+1}\right \vert }{2}+\frac{d-j}{4}\right).
\end{multline*}
Here, the beta function $B\left(  a,b\right)$ and the continuous Hahn polynomials $p_{k}\left( x;a,b,c,d\right)$ are given by the definitions provided in (\ref{beta}) and (\ref{hahn}), respectively.
\end{lemma}

As an application of this Fourier transform and the Parseval's identity, the next lemma follows given in \cite{233}.

\begin{lemma}\cite{233}
Let $\mathbf{k}$ and $\mathbf{k}^{j}$ be defined as in \eqref{notation}, let
$\mathbf{\alpha=}\left(  \alpha_{1},\alpha_{2}\right)  $ and $\left \vert \mathbf{\alpha}%
\right \vert =\alpha_{1}+\alpha_{2}$. Then, we have
\begin{multline}
 \int \limits_{-\infty}^{\infty}\cdots \int \limits_{-\infty}^{\infty}%
\ D_{\boldsymbol{k}}\left(  i\boldsymbol{x};\alpha_{1},\alpha_{2}\right)
\ D_{\boldsymbol{k'}}\left(  -i\boldsymbol{x};\alpha_{2},\alpha_{1}\right)
\mathbf{dx}
 =\left(  2\pi \right)  ^{d}2^{-2d\left \vert \mathbf{\alpha}\right \vert
+d+1}h_{\mathbf{k}}^{\left(  \alpha_{1}+\alpha_{2}-\frac{1}{2}\right)  }\\
 \times \prod \limits_{j=1}^{d}\frac{\left(  k_{j}!\right)  ^{2}\Gamma \left(
\left \vert \mathbf{k}^{j+1}\right \vert +2\alpha_{1}+\frac{d-j}{2}\right)
\Gamma \left(  \left \vert \mathbf{k}^{j+1}\right \vert +2\alpha_{2}+\frac{d-j}%
{2}\right)  }{2^{2\left \vert \mathbf{k}^{j+1}\right \vert }\left(  \left(
2\left \vert \mathbf{k}^{j+1}\right \vert +2\left \vert \mathbf{\alpha}\right \vert
+d-j-1\right)  _{k_{j}}\right)  ^{2}}\delta_{k_{j},k_{j}'},%
\end{multline}
for $\alpha_{1},\alpha_{2}>0$ where $h_{\mathbf{k}}^{\left(  \alpha_{1}+\alpha_{2}-\frac{1}%
{2}\right)  }$ is given in \eqref{Norm} and%
\begin{multline}
 \ D_{\mathbf{k}}\left(  \boldsymbol{x};\alpha_{1},\alpha_{2}\right)
=\prod \limits_{j=1}^{d}\left \{  \Gamma \left(  \alpha_{1}+\frac{\left \vert
\mathbf{k}^{j+1}\right \vert -x_{j}}{2}+\frac{d-j}{4}\right)  \Gamma \left(
\alpha_{1}+\frac{\left \vert \mathbf{k}^{j+1}\right \vert +x_{j}}{2}+\frac{d-j}%
{4}\right) \right. \\
 \left.\times\hyper{3}{2}{-k_{j},\ k_{j}+2\left(  \left \vert \mathbf{k}%
^{j+1}\right \vert +\left \vert \mathbf{\alpha}\right \vert +\frac{d-j-1}{2}\right)
,\alpha_{1}+\frac{\left \vert \mathbf{k}^{j+1}\right \vert +x_{j}}{2}+\frac{d-j}%
{4}}{\left \vert \mathbf{k}^{j+1}\right \vert +\left \vert \mathbf{\alpha}\right \vert
+\frac{d-j}{2},\  \left \vert \mathbf{k}^{j+1}\right \vert +2\alpha_{1}+\frac{d-j}{2}}{1} \right \} \label{ball},
\end{multline}

or, expressed using the continuous Hahn polynomials \eqref{hahn}
\begin{multline}
\ D_{\mathbf{k}}\left(  \boldsymbol{x};\alpha_{1},\alpha_{2}\right)  =%
{\displaystyle \prod \limits_{j=1}^{d}}
\left \{  \dfrac{k_{j}!i^{-k_{j}}}{\left(  \left \vert \mathbf{k}^{j+1}%
\right \vert +2\alpha_{1}+\dfrac{d-j}{2}\right)  _{k_{j}}\left(  \left \vert
\mathbf{k}^{j+1}\right \vert +\left \vert \mathbf{\alpha}\right \vert +\dfrac{d-j}%
{2}\right)  _{k_{j}}}\right.  \\
\times \left.  \Gamma \left(  \alpha_{1}+\dfrac{\left \vert \mathbf{k}^{j+1}%
\right \vert -x_{j}}{2}+\dfrac{d-j}{4}\right)  \Gamma \left(  \alpha_{1}%
+\dfrac{\left \vert \mathbf{k}^{j+1}\right \vert +x_{j}}{2}+\dfrac{d-j}%
{4}\right)  \right.  \\
\times p_{k_{j}}\left(  -\dfrac{ix_{j}}{2};\alpha_{1}+\dfrac{\left \vert
\mathbf{k}^{j+1}\right \vert }{2}+\dfrac{d-j}{4},\alpha_{2}+\dfrac{\left \vert
\mathbf{k}^{j+1}\right \vert }{2}+\dfrac{d-j}{4}\right.  \\
 \left.  \left.  ,\alpha_{2}+\dfrac{\left \vert \mathbf{k}^{j+1}\right \vert
}{2}+\dfrac{d-j}{4},\alpha_{1}+\dfrac{\left \vert \mathbf{k}^{j+1}\right \vert }%
{2}+\dfrac{d-j}{4}\right)  \right \},
\end{multline}
for $d\geq1$.
\end{lemma}

\section{The Fourier Transforms of Orthogonal Polynomials on the Paraboloid}\label{sec:ft}

This study focuses primarily on determining the Fourier transforms of the Laguerre and Jacobi polynomials defined on the paraboloid
${\mathbb{U}}^{d+1}$. These transforms are shown to be representable through continuous Hahn polynomials. Furthermore, the analysis leads to the introduction of new families of multivariate orthogonal functions expressed in terms of multivariate Hahn polynomials. The discussion proceeds in two stages: first, addressing the Jacobi case on the paraboloid, followed by the Laguerre case, both following an analogous structural framework. In each setting, the corresponding Fourier transform is derived, after which the new function classes emerge via the application of the Fourier transform together with Parseval's identity. Then, some contiguous relations for these function classes are given. The proofs of these results are provided in the following section.

\subsection{The Fourier transform of Jacobi polynomials on the paraboloid}
Let us focus on a function formulated using Jacobi polynomials on the paraboloid \eqref{eq:sQcone} as
\begin{multline}
h_{m,\mathbf{k}}\left( t, \boldsymbol{x};\alpha,\zeta,\eta,\beta,\gamma,\mu \right)    :=h_{m,k_{1},\dots,k_{d}}\left(
t,x_{1},\dots,x_{d};\alpha,\zeta,\eta,\beta,\gamma,\mu \right)  \\
 =\prod \limits_{j=1}^{d}\left(  1-\tanh^{2}x_{j}\right)  ^{\alpha+\frac{d-j}{4}%
}(1+\tanh t)^{\zeta}(1-\tanh t)^{\eta}{\mathsf{Q}}_{{\mathbf{k}},m}^{n}\left(  \varsigma_{1},\dots,\varsigma_{d},\varsigma_{d+1}\right),
\end{multline}
for $d\geq1$, where $\alpha,\zeta,\eta,\beta,\mu,\gamma$ real parameters and
\begin{align*}
\varsigma_{1}\left( t \right) & =\varsigma_{1}=\left( \frac{1+\tanh t}{2}\right) \\
\varsigma_{2}\left( t, x_{1}\right)   & =\varsigma_{2}={\left( \frac{1+\tanh t}{2}\right)}^{1/2}\vartheta_{1}\left( x_{1}\right) ,\\
\varsigma_{3}\left( t, x_{1},x_{2}\right)   & =\varsigma_{3}={\left( \frac{1+\tanh t}{2}\right)}^{1/2}\vartheta_{2}\left( x_{1},x_{2}\right),\\
\vdots\\
\varsigma_{d+1}\left( t, x_{1},\dots,x_{d}\right)   & =\varsigma_{d+1}={\left( \frac{1+\tanh t}{2}\right)}^{1/2}\vartheta_{d}\left( x_{1},\dots,x_{d}\right) ,
\end{align*}
for $d\geq1$ where $\vartheta_{1},...,\vartheta_{d}$ are defined by \eqref{var1}.
Explicitly, this can be expressed as
\begin{multline*}
h_{m,\mathbf{k}}\left( t, \boldsymbol{x};\alpha,\zeta,\eta,\beta,\gamma,\mu \right)  =\frac{1}{2^{\frac{\left \vert \mathbf{k}\right \vert}{2}}}(1+\tanh t)^{\zeta+{\left \vert \mathbf{k}\right \vert}/2}(1-\tanh t)^{\eta}\ P_{m-\left \vert \mathbf{k}\right \vert }^{(\left \vert \mathbf{k}\right \vert+\mu+\beta+\frac{d-1}{2},\gamma)}\left(-\tanh t\right)\\
\times \prod \limits_{j=1}^{d}\left(  1-\tanh^{2}x_{j}\right)  ^{\alpha+\frac{d-j}{4}} \prod \limits_{j=1}^{d-1}\left(  1-\tanh^{2}x_{j}\right)  ^\frac{k_{j+1}+\cdots+k_{d}}{2}
\prod \limits_{j=1}^{d}C_{k_{j}}^{\left(  \lambda_{j} \right)}\left(\tanh x_{j}\right)
\end{multline*}
where $\lambda_{j}=\mu+\left \vert \mathbf{k}^{j+1}\right \vert +\frac{d-j}{2}$, or in terms of the function  $g_{d}\left(x_{1},\dots,x_{d};k_{1},\dots,k_{d},\alpha,\mu \right)$ defined  in \eqref{150}
\begin{multline}
h_{m,\mathbf{k}}\left( t, \boldsymbol{x};\alpha,\zeta,\eta,\beta,\gamma,\mu \right) =\frac{1}{2^{{\left \vert \mathbf{k}\right \vert}/2}}(1+\tanh t)^{\zeta+{\left \vert \mathbf{k}\right \vert}/2}(1-\tanh t)^{\eta}\\ \times P_{m-\left \vert \mathbf{k}\right \vert }^{(\left \vert \mathbf{k}\right \vert+\mu+\beta+\frac{d-1}{2},\gamma)}\left(-\tanh t\right)
 g_{d}\left(
x_{1},\dots,x_{d};k_{1},\dots,k_{d},\alpha,\mu \right).\label{jac-spec-func}
\end{multline}

\medskip
Based on the notations introduced above, we proceed to present the Fourier transform of the function $h_{m,\mathbf{k}}\left( t, \boldsymbol{x};\alpha,\zeta,\eta,\beta,\gamma,\mu \right) $ defined by (\ref{jac-spec-func}).

\medskip
\begin{theorem}\label{theorem:33}
The explicit form of the Fourier transform of the function $h_{m,\mathbf{k}}\left( t, \boldsymbol{x};\alpha,\zeta,\eta,\beta,\gamma,\mu \right)$ is stated as follows
\begin{multline}%
\mathcal{F}
\left( h_{m,\mathbf{k}}\left( t, \boldsymbol{x};\alpha,\zeta,\eta,\beta,\gamma,\mu \right) \right)  :=%
\mathcal{F}
\left( h_{m,\mathbf{k}}\left( t, x_{1},\dots,x_{d};\alpha,\zeta,\eta,\beta,\gamma,\mu\right)  \right) \\
=2^{\zeta+\eta-1}\frac{ \left(  \left\vert \mathbf{k} \right\vert + \mu + \beta + \frac{d+1}{2} \right)_{m - \left\vert \mathbf{k} \right\vert} \Gamma \left( \zeta + \frac{\left\vert \mathbf{k} \right\vert }{2}- \frac{i\xi_{d+1}}{2} \right)\Gamma \left( \eta + \frac{i\xi_{d+1}}{2} \right)}{\left( m- \left\vert \mathbf{k} \right\vert \right)!\Gamma \left( \frac{\left\vert \mathbf{k} \right\vert}{2}+\zeta+\eta\right)}   \\
\times \Theta \left(  m,\mathbf{k},\zeta,\eta,\beta,\gamma,\mu,\xi_{d+1}\right)\mathcal{F} \left(  g_{d}\left( \boldsymbol{x};\mathbf{k},\alpha,\mu \right) \right)(\xi_{1},\dots,\xi_{d})  \\
=2^{\zeta+\eta-1+2\alpha d+\frac{d\left(  d-5\right)  }{4}+\text{$
{\textstyle \sum \limits_{j=1}^{d-1}}
$}{ jk}_{j+1}}\frac{ \left(  \left\vert \mathbf{k} \right\vert + \mu + \beta + \frac{d+1}{2} \right)_{m - \left\vert \mathbf{k} \right\vert} \Gamma \left( \zeta + \frac{\left\vert \mathbf{k} \right\vert }{2}- \frac{i\xi_{d+1}}{2} \right)\Gamma \left( \eta + \frac{i\xi_{d+1}}{2} \right)}{\left( m- \left\vert \mathbf{k} \right\vert \right)!\Gamma \left( \frac{\left\vert \mathbf{k} \right\vert}{2}+\zeta+\eta\right)}   \\
\times \prod \limits_{j=1}^{d}\left \{  \frac{\left(  2\left(
\left \vert \mathbf{k}^{j+1}\right \vert +\mu+\frac{d-j}{2}\right)  \right)
_{k_{j}}}{k_{j}!} \varphi_{j}^{d}\left(  \alpha,\mu,\mathbf{k};\xi_{j}\right)
\right \} \Theta \left(  m,\mathbf{k},\zeta,\eta,\beta,\gamma,\mu,\xi_{d+1}\right)  \label{Jac-Fourier}%
\end{multline}
where $\varphi_{j}^{d}\left( \alpha,\mu,\mathbf{k};\xi_{j}\right)$ is defined as in Lemma 2.1 and
\begin{align*}
\Theta\left(  m,\mathbf{k},\zeta,\eta,\beta,\gamma,\mu,\xi_{d+1}\right) = \hyper{3}{2}{-m+\left \vert \mathbf{k}\right \vert,\ m+\mu+\beta+\gamma+\frac{d+1}{2},\ \frac{\left \vert \mathbf{k}\right \vert}{2}+\zeta -\frac{i\xi_{d+1}}{2}}{\left \vert \mathbf{k}\right \vert
+\mu+\beta+\frac{d+1}{2},\ \frac{\left \vert \mathbf{k}\right \vert}{2}+\zeta+\eta}{1}.
\end{align*}
\end{theorem}
\subsection{Special function classes obtained via the Fourier transform of Jacobi polynomials on the paraboloid}

In this subsection, using Parseval's identity together with the Fourier transform of the Jacobi polynomials on the paraboloid, we construct a family of special functions.

\begin{theorem}\label{theorem:34}
Considering $\mathbf{k}$ and $\mathbf{k}^{j}$ as in (\ref{notation}), let
$\mathbf{\alpha=}\left(  \alpha_{1},\alpha_{2}\right)  $, $\left \vert \mathbf{\alpha}\right \vert
=\alpha_{1}+\alpha_{2},~\mathbf{\zeta=}\left( \zeta_{1},\zeta_{2}\right)  $, $\left \vert
\mathbf{\zeta}\right \vert =\zeta_{1}+\zeta_{2}$, $\mathbf{\eta=}\left(  \eta_{1},\eta_{2}\right)  $
and $\left \vert \mathbf{\eta}\right \vert =\eta_{1}+\eta_{2}.$ One obtains the following relation
\begin{multline*}
  \int \limits_{-\infty}^{\infty}\cdots \int \limits_{-\infty}^{\infty}%
\ \Gamma \left(  \eta_{1}+\frac{it}{2}\right)  \Gamma \left( \eta_{2}%
-\frac{it}{2}\right)A_{m,\mathbf{k}}^{\left(  d+1\right)
}\left(  it,i\boldsymbol{x};\alpha_{1},\alpha_{2},\zeta_{1},\zeta_{2},\eta_{1},\eta_{2}\right) \\ \times A_{m',\mathbf{k'}}^{\left(  d+1\right)  }\left(
\boldsymbol{-}it,-i\boldsymbol{x};\alpha_{2},\alpha_{1},\zeta_{2},\zeta_{1},\eta_{2},\eta_{1}\right)
d\boldsymbol{x}dt\\
  =\pi ^{d+1}2^{-2d\left \vert \mathbf{\alpha}\right \vert
+2d+3}h_{\mathbf{k}}^{\left(  \alpha_{1}+\alpha_{2}-\frac{1}{2}\right)  }\left(
m-\left \vert \mathbf{k}\right \vert \right)  !\\
  \times \frac{\Gamma \left(  m+\left \vert
\mathbf{\zeta}\right \vert \right)  \Gamma \left(  m-\left \vert \mathbf{k}%
\right \vert +\left \vert \mathbf{\eta}\right \vert \right)  \Gamma \left(
\frac{\left \vert \mathbf{k}\right \vert}{2} +\zeta_{1}+\eta_{1}\right)  \Gamma \left( \frac{ \left \vert
\mathbf{k}\right \vert}{2} +\zeta_{2}+\eta_{2}\right)  }{\left(  \left( \left \vert
\mathbf{k}\right \vert +\left \vert \mathbf{\zeta}\right \vert \right)
_{m-\left \vert \mathbf{k}\right \vert }\right)  ^{2}\left(  2m- \left \vert
\mathbf{k}\right \vert+\left \vert
\mathbf{\zeta}\right \vert +\left \vert \mathbf{\eta}\right \vert -1\right)
\Gamma \left(  m +\left \vert \mathbf{\zeta}%
\right \vert +\left \vert \mathbf{\eta}\right \vert -1\right)  }\delta_{m,m'}\\
  \times \prod \limits_{j=1}^{d}\frac{\left(  k_{j}!\right)  ^{2}\Gamma \left(
\left \vert \mathbf{k}^{j+1}\right \vert +2\alpha_{1}+\frac{d-j}{2}\right)
\Gamma \left(  \left \vert \mathbf{k}^{j+1}\right \vert +2\alpha_{2}+\frac{d-j}%
{2}\right)  }{2^{2\left \vert \mathbf{k}^{j+1}\right \vert }\left(  \left(
2\left \vert \mathbf{k}^{j+1}\right \vert +2\left \vert \mathbf{\alpha}\right \vert
+d-j-1\right)  _{k_{j}}\right)  ^{2}}\delta_{k_{j},k_{j}'}%
\end{multline*}
for assuming $\alpha_{1},\alpha_{2},\zeta_{1},\zeta_{2},\eta_{1},\eta_{2}>0$ where $h_{\mathbf{k}}^{\left(
\alpha_{1}+\alpha_{2}-\frac{1}{2}\right)  }$ is expressed as (\ref{Norm}) and%
\begin{multline*}
 A_{m,\mathbf{k}}^{\left(  d+1\right)  }\left(
t,\boldsymbol{x};\alpha_{1},\alpha_{2},\zeta_{1},\zeta_{2},\eta_{1},\eta_{2}\right)  \\
=\Gamma\left( \frac{\left \vert \mathbf{k}\right \vert}{2}+ \zeta_{1}-\frac{t}{2}\right)\hyper{3}{2}{-m+\left \vert \mathbf{k}\right \vert ,\ m+\left \vert \mathbf{\zeta}\right \vert +\left \vert
\mathbf{\eta}\right \vert -1,\frac{\left \vert \mathbf{k}\right \vert}{2} +\zeta_{1}-\frac{t}
{2}}{\left \vert \mathbf{k}\right \vert +\left \vert \mathbf{\zeta}\right \vert
,\frac{\left \vert \mathbf{k}\right \vert}{2} +\zeta_{1}+\eta_{1}}{1}D_{\mathbf{k}}\left(  \boldsymbol{x};\alpha_{1},\alpha_{2}\right)
\end{multline*}
equivalently, via Hahn polynomials (\ref{hahn})%
\begin{multline*}
\  A_{m,\mathbf{k}}^{\left(  d+1\right)  }\left(
t,\boldsymbol{x};\alpha_{1},\alpha_{2},\zeta_{1},\zeta_{2},\eta_{1},\eta_{2}\right)  \\
=\dfrac{\left(  m-\left \vert
\mathbf{k}\right \vert \right)  !i^{\left \vert \mathbf{k}\right \vert -m}\Gamma\left( \frac{\left \vert \mathbf{k}\right \vert}{2}+ \zeta_{1}-\frac{t}{2}\right)}{\left(
\left \vert \mathbf{k}\right \vert +\left \vert \mathbf{\zeta}\right \vert \right)
_{m-\left \vert \mathbf{k}\right \vert }\left( \frac{ \left \vert \mathbf{k}\right \vert}{2}
+\zeta_{1}+\eta_{1}\right)  _{m-\left \vert \mathbf{k}\right \vert }}
 p_{m-\left \vert \mathbf{k}\right \vert }\left(  \dfrac{it}%
{2};\frac{\left \vert \mathbf{k}\right \vert}{2} +\zeta_{1},\eta_{2},\frac{\left \vert \mathbf{k}%
\right \vert }{2}+\zeta_{2},\eta_{1}\right)D_{\mathbf{k}}\left(  \boldsymbol{x};\alpha_{1},\alpha_{2}\right),
\end{multline*}
for $d\geq1$ where $D_{\mathbf{k}}\left( \boldsymbol{x};\alpha_{1},\alpha_{2}\right)$ is given by \eqref{ball}.
\end{theorem}
\subsection{Contiguous relations for the function $A_{m,\mathbf{k}}^{\left(d+1\right)}\left(t,\boldsymbol{x};\alpha_{1},\alpha_{2},\zeta_{1},\zeta_{2},\eta_{1},\eta_{2}\right)$}
We now give some contiguous relations for the function $A_{m,\mathbf{k}}^{\left(d+1\right)}\left(t,\boldsymbol{x};\alpha_{1},\alpha_{2},\zeta_{1},\zeta_{2},\eta_{1},\eta_{2}\right)$ in the following theorem.
\begin{theorem}\label{theorem:rec1}
The function $A_{m,\mathbf{k}}^{\left(  d+1\right)  }\left(  t,\boldsymbol{x}%
;\alpha_{1},\alpha_{2},\zeta_{1},\zeta_{2},\eta_{1},\eta_{2}\right)  $
satisfies the following relations
\begin{align*}
i) \ & \left(  m+\left \vert \zeta \right \vert +\left \vert \eta \right \vert -1\right)
A_{m,\mathbf{k}}^{\left(  d+1\right)  }\left(  t,\boldsymbol{x};\alpha_{1}%
,\alpha_{2},\zeta_{1},\zeta_{2},\eta_{1},\eta_{2}+1\right) \\
&  \ \ \ \quad +\left(  m-\left \vert \mathbf{k}\right \vert \right)  A_{m-1,\mathbf{k}%
}^{\left(  d+1\right)  }\left(  t,\boldsymbol{x};\alpha_{1},\alpha_{2},\zeta
_{1},\zeta_{2},\eta_{1},\eta_{2}+1\right) \\
&\ \ \ \ \ \ \ \ \ \ \  \ \ \ \quad  =\left(  2m-\left \vert \mathbf{k}\right \vert +\left \vert \zeta \right \vert
+\left \vert \eta \right \vert -1\right)  A_{m,\mathbf{k}}^{\left(  d+1\right)
}\left(  t,\boldsymbol{x};\alpha_{1},\alpha_{2},\zeta_{1},\zeta_{2},\eta_{1}%
,\eta_{2}\right)  ,
\end{align*}
\begin{align*}
ii) \ & \left(  \frac{\left \vert \mathbf{k}\right \vert }{2}+\zeta_{2}-\eta
_{1}\right)  A_{m,\mathbf{k}}^{\left(  d+1\right)  }\left(  t,\boldsymbol{x}%
;\alpha_{1},\alpha_{2},\zeta_{1},\zeta_{2}+1,\eta_{1}+1,\eta_{2}-2\right) \\
&\ \ \ \  \ \ \quad   +\left(  \frac{\left \vert \mathbf{k}\right \vert }{2}+\zeta_{1}+\eta
_{1}\right)  A_{m,\mathbf{k}}^{\left(  d+1\right)  }\left(  t,\boldsymbol{x}%
;\alpha_{1},\alpha_{2},\zeta_{1},\zeta_{2}+1,\eta_{1},\eta_{2}-1\right) \\
&\ \ \ \ \ \ \ \ \ \ \  \ \ \ \quad   =\left(  \left \vert \mathbf{k}\right \vert +\left \vert \zeta \right \vert
\right)  A_{m,\mathbf{k}}^{\left(  d+1\right)  }\left(  t,\boldsymbol{x}%
;\alpha_{1},\alpha_{2},\zeta_{1},\zeta_{2},\eta_{1}+1,\eta_{2}-1\right)  ,
\end{align*}
\begin{align*}
iii) \ & \left(  m-\left \vert \mathbf{k}\right \vert \right)  \left(  \frac{\left \vert
\mathbf{k}\right \vert }{2}+\zeta_{2}-\eta_{1}\right)  A_{m-1,\mathbf{k}%
}^{\left(  d+1\right)  }\left(  t,\boldsymbol{x};\alpha_{1},\alpha_{2},\zeta
_{1},\zeta_{2}+1,\eta_{1}+1,\eta_{2}-1\right) \\
& \ \ \ \ \ \ \ \quad  +\left(  \frac{\left \vert \mathbf{k}\right \vert }{2}+\zeta_{1}+\eta
_{1}\right)  \left(  m+\left \vert \zeta \right \vert \right)  A_{m,\mathbf{k}%
}^{\left(  d+1\right)  }\left(  t,\boldsymbol{x};\alpha_{1},\alpha_{2},\zeta
_{1},\zeta_{2}+1,\eta_{1},\eta_{2}-1\right) \\
&\ \ \ \ \ \ \ \ \ \ \  \ \ \ \quad   =\left(  \left \vert \mathbf{k}\right \vert +\left \vert \zeta \right \vert
\right)  \left(  -\frac{\left \vert \mathbf{k}\right \vert }{2}+\zeta_{1}%
+\eta_{1}+m\right)  A_{m,\mathbf{k}}^{\left(  d+1\right)  }\left(
t,\boldsymbol{x};\alpha_{1},\alpha_{2},\zeta_{1},\zeta_{2},\eta_{1}+1,\eta
_{2}-1\right)  ,
\end{align*}
\begin{align*}
iv) \ & \left(  \left \vert \mathbf{k}\right \vert -2m-\left \vert \zeta \right \vert
-\left \vert \eta \right \vert +1\right) \left(  \frac{\left \vert \mathbf{k}\right \vert }{2}+\zeta_{1}%
-\frac{t}{2}\right) A_{m-1,\mathbf{k}}^{\left(  d+1\right)
}\left(  t,\boldsymbol{x};\alpha_{1},\alpha_{2},\zeta_{1}+1,\zeta_{2},\eta
_{1},\eta_{2}+1\right) \\
&\ \ \ \ \ \ \ \ \ \quad   =\left(  \left \vert \mathbf{k}\right \vert +\left \vert \zeta \right \vert
\right)  \left(  \frac{\left \vert \mathbf{k}\right \vert }{2}+\zeta_{1}%
+\eta_{1}\right)  \left \{  A_{m,\mathbf{k}}^{\left(  d+1\right)  }\left(
t,\boldsymbol{x};\alpha_{1},\alpha_{2},\zeta_{1},\zeta_{2},\eta_{1},\eta
_{2}+1\right)  \right. \\
& \ \ \ \ \ \ \ \ \ \ \ \ \ \ \ \ \ \ \ \ \ \ \ \quad  -\left.  A_{m-1,\mathbf{k}}^{\left(  d+1\right)  }\left(  t,\boldsymbol{x}%
;\alpha_{1},\alpha_{2},\zeta_{1},\zeta_{2},\eta_{1},\eta_{2}+1\right)
\right \}  ,
\end{align*}
\begin{align*}
v) \ & \left(  m+\left \vert \zeta \right \vert +\left \vert \eta \right \vert -1\right)
\left(  m+\left \vert \zeta \right \vert \right)  A_{m,\mathbf{k}}^{\left(
d+1\right)  }\left(  t,\boldsymbol{x};\alpha_{1},\alpha_{2},\zeta_{1},\zeta
_{2}+1,\eta_{1},\eta_{2}\right) \\
&  \ \ \ \ \ \ \ \ \ \ \ \  \quad =\left(  -m+\left \vert \mathbf{k}\right \vert \right)  \left(  \left \vert
\mathbf{k}\right \vert -m-\left \vert \eta \right \vert +1\right)
A_{m-1,\mathbf{k}}^{\left(  d+1\right)  }\left(  t,\boldsymbol{x};\alpha
_{1},\alpha_{2},\zeta_{1},\zeta_{2}+1,\eta_{1},\eta_{2}\right) \\
&  \ \ \ \ \ \ \ \ \ \ \ \  \quad \quad \quad \quad +\left(  2m-\left \vert \mathbf{k}\right \vert +\left \vert \zeta \right \vert
+\left \vert \eta \right \vert -1\right)  \left(  \left \vert \mathbf{k}%
\right \vert +\left \vert \zeta \right \vert \right)  A_{m,\mathbf{k}}^{\left(
d+1\right)  }\left(  t,\boldsymbol{x};\alpha_{1},\alpha_{2},\zeta_{1},\zeta
_{2},\eta_{1},\eta_{2}\right)  ,
\end{align*}
\begin{align*}
vi) \ & \left(  m+\left \vert \zeta \right \vert \right)  A_{m,\mathbf{k}}^{\left(
d+1\right)  }\left(  t,\boldsymbol{x};\alpha_{1},\alpha_{2},\zeta_{1},\zeta
_{2}+1,\eta_{1},\eta_{2}-1\right) \\
& \ \ \ \ \ \ \ \ \ \  \quad+\left(  -m+\left \vert \mathbf{k}\right \vert \right)  A_{m-1,\mathbf{k}%
}^{\left(  d+1\right)  }\left(  t,\boldsymbol{x};\alpha_{1},\alpha_{2},\zeta
_{1},\zeta_{2}+1,\eta_{1},\eta_{2}\right) \\
&\ \ \ \ \ \ \ \ \ \  \quad  \quad  \quad  \quad =\left(  \left \vert \mathbf{k}\right \vert +\left \vert \zeta \right \vert
\right)  A_{m,\mathbf{k}}^{\left(  d+1\right)  }\left(  t,\boldsymbol{x}%
;\alpha_{1},\alpha_{2},\zeta_{1},\zeta_{2},\eta_{1},\eta_{2}\right)  ,
\end{align*}
\begin{align*}
vii) \ & A_{m,\mathbf{k}}^{\left(  d+1\right)  }\left(  t-2,\boldsymbol{x};\alpha
_{1},\alpha_{2},\zeta_{1},\zeta_{2},\eta_{1},\eta_{2}\right) \\
&  \ \ \ \ \ \ \quad=\left(  \frac{\left \vert \mathbf{k}\right \vert }{2}+\zeta_{1}-\frac{t}%
{2}\right)  A_{m,\mathbf{k}}^{\left(  d+1\right)  }\left(  t,\boldsymbol{x}%
;\alpha_{1},\alpha_{2},\zeta_{1},\zeta_{2},\eta_{1},\eta_{2}\right) \\
&  \ \ \ \ \ \ \ \ \  \quad \quad \quad \quad+\frac{\left(  -m+\left \vert \mathbf{k}\right \vert \right)  \left(
m+\left \vert \zeta \right \vert +\left \vert \eta \right \vert -1\right)  }{\left(
\left \vert \mathbf{k}\right \vert +\left \vert \zeta \right \vert \right)  \left(
\frac{\left \vert \mathbf{k}\right \vert }{2}+\zeta_{1}+\eta_{1}\right)
}A_{m-1,\mathbf{k}}^{\left(  d+1\right)  }\left(  t,\boldsymbol{x};\alpha
_{1},\alpha_{2},\zeta_{1}+1,\zeta_{2},\eta_{1},\eta_{2}+1\right)  .
\end{align*}

\end{theorem}

\subsection{The Fourier transform of Laguerre polynomials on the paraboloid}
Now, for the Laguerre case we proceed the same framework applied to the Jacobi case on the paraboloid. Let us focus on a function formulated using Laguerre polynomials on the paraboloid \eqref{eq:sQpara} as

\begin{multline}
h_{m,\mathbf{k}}\left( t,\boldsymbol{x};\alpha,\zeta,\beta,\mu \right)    :=h_{m,k_{1},\dots,k_{d}}\left(
t,x_{1},\dots,x_{d};\alpha,\zeta,\beta,\mu  \right)  \\
 =\prod \limits_{j=1}^{d}\left(  1-\tanh^{2}x_{j}\right)  ^{\alpha+\frac{d-j}{4}%
}{\mathsf{R}}_{{\mathbf{k}},m}^{n}\left(  \sigma_{1},\dots,\sigma_{d},\sigma_{d+1}\right)e^{-\frac{e^{t}}{2}+bt}
,\label{15}%
\end{multline}
for $d\geq1$, where $\alpha,\zeta,\beta,\mu$ are real parameters and
\begin{align*}
\sigma_{1}\left(  t\right)   & =\sigma_{1}=e^t,\\
\sigma_{2}\left( t, x_{1}\right)   & =\sigma_{2}=e^{t/2}\vartheta_{1}\left(  x_{1}\right) ,\\
\sigma_{3}\left( t, x_{1},x_{2}\right)   & =\sigma_{3}=e^{t/2}\vartheta_{2}\left(  x_{1},x_{2}\right) ,\\
\vdots\\
\sigma_{d+1}\left( t, x_{1},\dots,x_{d}\right)   & =\sigma_{d+1}=e^{t/2}\vartheta_{d}\left(  x_{1},\dots,x_{d}\right),
\end{align*}
for $d\geq1$ where $\vartheta_{1},...,\vartheta_{d}$ are defined by \eqref{var1}.
We may represent this function explicitly as
\begin{align*}
h_{m,\mathbf{k}}\left( t,\boldsymbol{x};\alpha,\zeta,\beta,\mu \right) & =\prod \limits_{j=1}^{d}\left(  1-\tanh^{2}x_{j}\right)  ^{\alpha+\frac{d-j}{4}%
}\ L_{m-\left \vert \mathbf{k}\right \vert }^{\left \vert \mathbf{k}\right \vert+\mu+\beta+\frac{d-1}{2}}\left(e^t\right) \\
&\times e^{-\frac{e^{t}}{2}+\zeta t+\frac{\left \vert \mathbf{k}\right \vert}{2} t}\prod \limits_{j=1}^{d-1}\left(  1-\tanh^{2}x_{j}\right)  ^\frac{k_{j+1}+\cdots +k_{d}}{2}
\prod \limits_{j=1}^{d}C_{k_{j}}^{\left(  \lambda_{j} \right)}\left(\tanh x_{j}\right)
\end{align*}
where $\lambda_{j}=\mu+\left \vert \mathbf{k}^{j+1}\right \vert +\frac{d-j}{2}$, equivalently
\begin{equation}
h_{m,\mathbf{k}}\left( t,\boldsymbol{x};\alpha,\zeta,\beta,\mu \right) \\
=e^{-\frac{e^{t}}{2}+\zeta t+\frac{\left \vert \mathbf{k}\right \vert}{2} t} L_{m-\left \vert \mathbf{k}\right \vert }^{\left \vert \mathbf{k}\right \vert+\mu+\beta+\frac{d-1}{2}}\left(e^t\right)g_{d}\left(
\boldsymbol{x};\mathbf{k},\alpha,\mu \right)\label{100}
\end{equation}
where $g_{d}\left(\boldsymbol{x};\mathbf{k},\alpha,\mu \right)$ is defined in (\ref{150}).

\begin{theorem}\label{theorem:31}
The explicit form of the Fourier transform of the function $h_{m,\mathbf{k}}\left( t,\boldsymbol{x};\alpha,\zeta,\beta,\mu \right)$ is stated as follows
\begin{multline}%
\mathcal{F}
\left( h_{m,\mathbf{k}}\left( t,\boldsymbol{x};\alpha,\zeta,\beta,\mu \right)   \right)  :=%
\mathcal{F}\left( h_{m,\mathbf{k}}\left( t,x_{1},\dots,x_{d};\alpha,\zeta,\beta,\mu \right) \right) \\
=2^{\zeta+\frac{\left \vert \mathbf{k}\right \vert}{2} -i\xi_{d+1}}\frac{ \left( \left\vert \mathbf{k} \right\vert + \mu + \beta + \frac{d+1}{2} \right)_{m - \left\vert \mathbf{k} \right\vert} \Gamma \left( \zeta+\frac{ \left\vert \mathbf{k} \right\vert}{2} - i \xi_{d+1} \right)}{\left( m - \left\vert \mathbf{k} \right\vert \right)!}   \\
\times\Lambda \left(  m,\mathbf{k},\zeta,\mu,\beta ,\xi_{d+1}\right) \mathcal{F} \left(  g_{d}\left( \boldsymbol{x};\mathbf{k},\alpha,\mu \right) \right)(\xi_{1},\dots,\xi_{d})\\
=2^{\zeta+\frac{\left \vert \mathbf{k}\right \vert}{2} -i\xi_{d+1}+2\alpha d+\frac{d\left(  d-5\right)  }{4}+\text{$
{\textstyle \sum \limits_{j=1}^{d-1}}
$}{ jk}_{j+1}}\frac{ \left( \left\vert \mathbf{k} \right\vert + \mu + \beta + \frac{d+1}{2} \right)_{m - \left\vert \mathbf{k} \right\vert} \Gamma \left( \zeta +\frac{ \left\vert \mathbf{k} \right\vert}{2} - i \xi_{d+1} \right)}{\left( m - \left\vert \mathbf{k} \right\vert \right)!}   \\
\times \prod \limits_{j=1}^{d}\left \{  \frac{\left(  2\left(
\left \vert \mathbf{k}^{j+1}\right \vert +\mu+\frac{d-j}{2}\right)  \right)
_{k_{j}}}{k_{j}!} \varphi_{j}^{d}\left(  \alpha,\mu,\mathbf{k};\xi_{j}\right)
\right \} \Lambda \left(  m,\mathbf{k},\zeta,\mu,\beta ,\xi_{d+1}\right)\label{18b}%
\end{multline}
where $\varphi_{j}^{d}\left(  \alpha,\mu,\mathbf{k};\xi_{j}\right)$ is expressed as in Lemma 2.1 and
\begin{align*}
\Lambda \left(  m,\mathbf{k},\zeta,\mu,\beta ,\xi_{d+1}\right) = \hyper{2}{1}{-m+\left \vert \mathbf{k}\right \vert ,\zeta+\frac{\left \vert \mathbf{k}\right \vert}{2} -i\xi_{d+1}}{\left \vert \mathbf{k}\right \vert +\mu+\beta+\frac{d+1}{2}}{2}.
\end{align*}
\end{theorem}

\subsection{Special function classes obtained via the Fourier transform of Laguerre polynomials on the paraboloid.}

We now proceed by applying Parseval's identity together with the Fourier transform of the Laguerre polynomials on the paraboloid, which leads us to a new class of special functions.

\begin{theorem}\label{theorem:32}
Denote $\mathbf{k}$ and $\mathbf{k}^{j}$ as in (\ref{notation}). Let
$\mathbf{\alpha=}\left(  \alpha_{1},\alpha_{2}\right)  $, $\left \vert \mathbf{\alpha}\right \vert
=\alpha_{1}+\alpha_{2},~\mathbf{\zeta=}\left(  \zeta_{1},\zeta_{2}\right)  $ and $\left \vert
\mathbf{\zeta}\right \vert =\zeta_{1}+\zeta_{2}$. The equality below holds

\begin{multline*}
  \int \limits_{-\infty}^{\infty}\cdots \int \limits_{-\infty}^{\infty}%
\ B_{m,\mathbf{k}}^{\left(  d+1\right)  }\left( it, i\boldsymbol{x,};\alpha_{1},\alpha_{2},\zeta_{1},\zeta_{2}\right) \ B_{m',\mathbf{k'}}^{\left(  d+1\right)
}\left( -it, -i\boldsymbol{x};\alpha_{2},\alpha_{1},\zeta_{2},\zeta_{1}\right)  d\boldsymbol{x}dt\\
  =\left(  2\pi \right)  ^{d+1}2^{-2d\left \vert \mathbf{\alpha}\right \vert
-\left \vert \mathbf{k}\right \vert -\left \vert \mathbf{\zeta}\right \vert
+d+1}h_{\mathbf{k}}^{\left(  \alpha_{1}+\alpha_{2}-\frac{1}{2}\right)}
\frac{\Gamma \left(  \left \vert \mathbf{\zeta}\right \vert +m \right) \left(  m-\left \vert \mathbf{k}\right \vert \right)
!}{(\left(  \left \vert \mathbf{k}\right \vert +\left \vert \mathbf{\zeta}\right \vert
\right)  _{m-\left \vert \mathbf{k}\right \vert })^{2}}\delta_{m,m'}\\
\times \prod \limits_{j=1}^{d}\frac{\left(  k_{j}!\right)  ^{2}\Gamma \left(
\left \vert \mathbf{k}^{j+1}\right \vert +2\alpha_{1}+\frac{d-j}{2}\right)
\Gamma \left(  \left \vert \mathbf{k}^{j+1}\right \vert +2\alpha_{2}+\frac{d-j}%
{2}\right)  }{2^{2\left \vert \mathbf{k}^{j+1}\right \vert }\left(  \left(
2\left \vert \mathbf{k}^{j+1}\right \vert +2\left \vert \mathbf{\alpha}\right \vert
+d-j-1\right)  _{k_{j}}\right)  ^{2}}\delta_{k_{j},k_{j}'}%
\end{multline*}
for $\alpha_{1},\alpha_{2},\zeta_{1},\zeta_{2}>0$ where 
\begin{align*}
\ B_{m,\mathbf{k}}^{\left(  d+1\right)  }\left( t, \boldsymbol{x}%
;\alpha_{1},\alpha_{2},\zeta_{1},\zeta_{2}\right) =\Gamma \left(  \zeta_{1}+\frac{\left \vert
\mathbf{k}\right \vert}{2}-t \right)D_{\mathbf{k}}\left( \boldsymbol{x};\alpha_{1},\alpha_{2}\right)\\
\times \hyper{2}{1}{-m+\left \vert \mathbf{k}\right \vert,\ \zeta_{1}+\frac{\left \vert
\mathbf{k}\right \vert}{2} -t}{\left \vert \mathbf{k}\right \vert +\left \vert
\mathbf{\zeta}\right \vert }{2}
\end{align*}
for $d\geq1$, $h_{\mathbf{k}}^{\left(  \alpha_{1}+\alpha_{2}-\frac{1}{2}\right) }$ and $D_{\mathbf{k}}\left( \boldsymbol{x};\alpha_{1},\alpha_{2}\right)$ are given by (\ref{Norm}) and \eqref{ball},respectively.
\end{theorem}

\subsection{Contiguous relations for the function $B_{m,\mathbf{k}}^{\left(  d+1\right)  }\left( t, \boldsymbol{x}%
;\alpha_{1},\alpha_{2},\zeta_{1},\zeta_{2}\right) $}

Here, we present some contiguous relations for $B_{m,\mathbf{k}}^{\left(  d+1\right)  }\left(  t,\boldsymbol{x}%
;\alpha_{1},\alpha_{2},\zeta_{1},\zeta_{2}\right)  $.
\begin{theorem}\label{theorem:rec2}
The function $B_{m,\mathbf{k}}^{\left(  d+1\right)  }\left(  t,\boldsymbol{x}%
;\alpha_{1},\alpha_{2},\zeta_{1},\zeta_{2}\right)  $ satisfies the following relations
\begin{align*}
i) \ &  \left(  \left \vert \mathbf{k}\right \vert +2\zeta_{2}+2t\right)
B_{m,\mathbf{k}}^{\left(  d+1\right)  }\left(  t+1,\boldsymbol{x};\alpha
_{1},\alpha_{2},\zeta_{1}+1,\zeta_{2}\right) \\
& \ \ \ \ \ \ \ \ \ \ \  \ \ \ \quad=\left(  \left \vert \mathbf{k}\right \vert +\left \vert \zeta \right \vert
\right)  \left \{  B_{m,\mathbf{k}}^{\left(  d+1\right)  }\left(
t,\boldsymbol{x};\alpha_{1},\alpha_{2},\zeta_{1},\zeta_{2}\right)
+B_{m+1,\mathbf{k}}^{\left(  d+1\right)  }\left(  t,\boldsymbol{x};\alpha
_{1},\alpha_{2},\zeta_{1},\zeta_{2}\right)  \right \}  ,
\end{align*}
\begin{align*}
ii) \ & 2\left(  \left \vert \zeta \right \vert +m\right)  B_{m,\mathbf{k}}^{\left(
d+1\right)  }\left(  t,\boldsymbol{x};\alpha_{1},\alpha_{2},\zeta_{1},\zeta
_{2}+1\right)  -\left(  \left \vert \mathbf{k}\right \vert +\left \vert
\zeta \right \vert \right)  B_{m,\mathbf{k}}^{\left(  d+1\right)  }\left(
t,\boldsymbol{x};\alpha_{1},\alpha_{2},\zeta_{1},\zeta_{2}\right) \\
& \ \ \ \ \ \ \ \ \ \ \  \ \ \ \quad=\left(  \left \vert \mathbf{k}\right \vert +\left \vert \zeta \right \vert
\right)  \left(  \frac{\left \vert \mathbf{k}\right \vert }{2}+\zeta
_{1}-t-1\right)  B_{m,\mathbf{k}}^{\left(  d+1\right)  }\left(  t+1,\boldsymbol{x}%
;\alpha_{1},\alpha_{2},\zeta_{1},\zeta_{2}\right)  ,
\end{align*}
\begin{align*}
iii) \ & \left(  m-\left \vert \mathbf{k}\right \vert \right)  B_{m-1,\mathbf{k}%
}^{\left(  d+1\right)  }\left(  t,\boldsymbol{x};\alpha_{1},\alpha_{2},\zeta
_{1},\zeta_{2}\right)  +B_{m,\mathbf{k}}^{\left(  d+1\right)  }\left(
t,\boldsymbol{x};\alpha_{1},\alpha_{2},\zeta_{1}+1,\zeta_{2}-1\right) \\
& \ \ \ \ \ \ \ \ \ \ \  \ \ \ \quad=\left(  m-\frac{\left \vert \mathbf{k}\right \vert }{2}+\zeta_{1}-t\right)
B_{m,\mathbf{k}}^{\left(  d+1\right)  }\left(  t,\boldsymbol{x};\alpha_{1}%
,\alpha_{2},\zeta_{1},\zeta_{2}\right)  ,
\end{align*}
\begin{align*}
iv) \ & \left(  m+\left \vert \zeta \right \vert -1\right)  B_{m,\mathbf{k}}^{\left(
d+1\right)  }\left(  t,\boldsymbol{x};\alpha_{1},\alpha_{2},\zeta_{1},\zeta
_{2}\right)  +\left(  -m+\left \vert \mathbf{k}\right \vert \right)
B_{m-1,\mathbf{k}}^{\left(  d+1\right)  }\left(  t,\boldsymbol{x};\alpha
_{1},\alpha_{2},\zeta_{1},\zeta_{2}\right) \\
& \ \ \ \ \ \ \ \ \ \ \  \ \ \ \quad=\left(  \left \vert \mathbf{k}\right \vert +\left \vert \zeta \right \vert
-1\right)  B_{m,\mathbf{k}}^{\left(  d+1\right)  }\left(  t,\boldsymbol{x}%
;\alpha_{1},\alpha_{2},\zeta_{1},\zeta_{2}-1\right)  ,
\end{align*}
\begin{align*}
v) \ & B_{m,\mathbf{k}}^{\left(  d+1\right)  }\left(  t,\boldsymbol{x};\alpha
_{1},\alpha_{2},\zeta_{1}+1,\zeta_{2}-1\right)  +\left(  \frac{\left \vert
\mathbf{k}\right \vert }{2}+\zeta_{2}+t-1\right)  B_{m,\mathbf{k}}^{\left(
d+1\right)  }\left(  t,\boldsymbol{x};\alpha_{1},\alpha_{2},\zeta_{1},\zeta
_{2}\right) \\
& \ \ \ \ \ \ \ \ \ \ \  \ \ \ \quad=\left(  \left \vert \mathbf{k}\right \vert +\left \vert \zeta \right \vert
-1\right)  B_{m,\mathbf{k}}^{\left(  d+1\right)  }\left(  t,\boldsymbol{x}%
;\alpha_{1},\alpha_{2},\zeta_{1},\zeta_{2}-1\right)  ,
\end{align*}
\begin{align*}
vi) \ & \left(  \zeta_{1}-\zeta_{2}-2t\right)  B_{m,\mathbf{k}}^{\left(  d+1\right)
}\left(  t,\boldsymbol{x};\alpha_{1},\alpha_{2},\zeta_{1},\zeta_{2}\right)
+\left(  m+\left \vert \zeta \right \vert \right)  B_{m+1,\mathbf{k}}^{\left(
d+1\right)  }\left(  t,\boldsymbol{x};\alpha_{1},\alpha_{2},\zeta_{1},\zeta
_{2}\right) \\
&\ \ \ \ \ \ \ \ \ \ \  \ \ \ \quad =\left(  m-\left \vert \mathbf{k}\right \vert \right)  B_{m-1,\mathbf{k}%
}^{\left(  d+1\right)  }\left(  t,\boldsymbol{x};\alpha_{1},\alpha_{2},\zeta
_{1},\zeta_{2}\right)  ,
\end{align*}
\begin{align*}
vii) \ & \left(  2\zeta_{2}+2t+m\right)  B_{m-1,\mathbf{k}}^{\left(  d+1\right)
}\left(  t,\boldsymbol{x};\alpha_{1},\alpha_{2},\zeta_{1},\zeta_{2}+1\right)
-\left(  m+\left \vert \zeta \right \vert \right)  B_{m,\mathbf{k}}^{\left(
d+1\right)  }\left(  t,\boldsymbol{x};\alpha_{1},\alpha_{2},\zeta_{1},\zeta
_{2}+1\right) \\
& \ \ \ \ \ \ \ \ \ \ \  \ \ \ \quad=\left(  \left \vert \mathbf{k}\right \vert +\left \vert \zeta \right \vert
\right)  B_{m-1,\mathbf{k}}^{\left(  d+1\right)  }\left(  t,\boldsymbol{x}%
;\alpha_{1},\alpha_{2},\zeta_{1},\zeta_{2}\right)  .
\end{align*}

\end{theorem}

\section{Proofs of the main results}

In what follows, we proceed by giving a detailed account of the proofs of the main results.

\begin{proof}[Proof of Theorem \ref{theorem:33}]
It is derived from (\ref{jac-spec-func})
\begin{multline*}
\mathcal{F}\left(h_{m,\mathbf{k}}\left( t, \boldsymbol{x};\alpha,\zeta,\eta,\beta,\gamma,\mu \right) \right)   \\
={\displaystyle \int \limits_{-\infty}^{\infty}}
{\displaystyle \int \limits_{-\infty}^{\infty}}\cdots {\displaystyle \int \limits_{-\infty}^{\infty}}
h_{m,\mathbf{k}}\left( t, \boldsymbol{x};\alpha,\zeta,\eta,\beta,\gamma,\mu \right) e^{-i\left(  \xi_{1}x_{1}+\cdots +\xi_{d}x_{d}+\xi_{d+1}t\right)  }d\boldsymbol{x}dt  \\
    ={\displaystyle \int \limits_{0}^{1}}
{\displaystyle \int \limits_{-\infty}^{\infty}}\cdots {\displaystyle \int \limits_{-\infty}^{\infty}}
e^{-i\left(  \xi_{1}x_{1}+\cdots +\xi_{d}x_{d}\right)}g_{d}\left(
x_{1},\dots,x_{d};k_{1},\dots,k_{d},\alpha,\mu \right)   \\
   \times 2^{\zeta+\eta-1} P_{m-\left \vert \mathbf{k}\right \vert }^{(\left \vert \mathbf{k}\right \vert+\mu+\beta+\frac{d-1}{2},\gamma)}\left(1-2u\right) u^{\zeta+\frac{\left \vert \mathbf{k}\right \vert}{2}-\frac{i\xi_{d+1}}{2}-1}(1-u)^{\eta+\frac{i\xi_{d+1}}{2}-1}d\boldsymbol{x}du \\
   =\mathcal{F}\left(g_{d}\left(
x_{1},\dots,x_{d};k_{1},\dots,k_{d},\alpha,\mu \right)\right)\nonumber\\
  \times 2^{\zeta+\eta-1}{\displaystyle \int \limits_{0}^{1}} P_{m-\left \vert \mathbf{k}\right \vert }^{(\left \vert \mathbf{k}\right \vert+\mu+\beta+\frac{d-1}{2},\gamma)}\left(1-2u\right) u^{\zeta+\frac{\left \vert \mathbf{k}\right \vert}{2}-\frac{i\xi_{d+1}}{2}-1}(1-u)^{\eta+\frac{i\xi_{d+1}}{2}-1}du.
\end{multline*}
Based on the definition of Jacobi polynomials (\ref{200}) and Gamma function, one can write
\begin{multline}
\mathcal{F} \left( h_{m,\mathbf{k}}\left( t, \boldsymbol{x};\alpha,\zeta,\eta,\beta,\gamma,\mu \right) \right)  = \mathcal{F}\left(g_{d}\left(
x_{1},\dots,x_{d};k_{1},\dots ,k_{d},\alpha,\mu \right)\right) \\
\times \frac{2^{\zeta+\eta-1}\left( \left \vert \mathbf{k}\right \vert +\mu+\beta+\frac{d+1}{2}\right)_{m-\left \vert \mathbf{k}\right \vert }}{\left(  m-\left \vert \mathbf{k}
\right \vert \right)  !}
  {\displaystyle \sum \limits_{j=0}^{m-\left \vert \mathbf{k}\right \vert }}
\frac{\left(  -m+\left \vert \mathbf{k}\right \vert \right)  _{j}\left(  m+\mu+\beta+\gamma+\frac{d+1}{2}\right)
_{j}}{\left(
\left \vert \mathbf{k}\right \vert +\mu+\beta+\frac{d+1}{2}\right)  _{j}~j!}\\
   {\displaystyle \times \int \limits_{0}^{1}}
u^{j+\zeta+\frac{\left \vert \mathbf{k}\right \vert}{2}-1-\frac{i\xi_{d+1}}{2}}(1-u)^{\eta-1+\frac{i\xi_{d+1}}{2}}du \nonumber \\
   =\mathcal{F}\left(g_{d}\left(
x_{1},\dots,x_{d};k_{1},\dots,k_{d},\alpha,\mu \right)\right)2^{\zeta+\eta-1}\left(  \left \vert \mathbf{k}\right \vert +\mu+\beta+\frac{d+1}{2}\right)
_{m-\left \vert \mathbf{k}\right \vert } \\
   \times \frac{\Gamma\left(  \frac{\left \vert \mathbf{k}\right \vert}{2} +\zeta-\frac{i\xi_{d+1}}{2} \right)\Gamma\left( \eta+\frac{i\xi_{d+1}}{2} \right)
}{\left(  m-\left \vert \mathbf{k}%
\right \vert \right)  !\Gamma\left( \frac{ \left \vert \mathbf{k}\right \vert}{2} +\zeta+\eta \right)}  \\
   \times \hyper{3}{2}{-m+\left \vert \mathbf{k}\right \vert,\ m+\mu+\beta+\gamma+\frac{d+1}{2},\ \frac{\left \vert \mathbf{k}\right \vert}{2}+\zeta -\frac{i\xi_{d+1}}{2}}{\left \vert \mathbf{k}\right \vert
+\mu+\beta+\frac{d+1}{2},\ \frac{\left \vert \mathbf{k}\right \vert}{2}+\zeta+\eta}{1}.
\end{multline}
The proof follows directly from equation (\ref{18}).

Let us now focus on a special case of the general result. For $d=1$, the special function is given by

\begin{align}
h_{m,k_{1}}\left( t, x_{1};\alpha,\zeta,\eta,\beta,\gamma,\mu \right)    &  =\left(  1-\tanh^{2}x_{1}\right)^{\alpha}(1+\tanh t)^{\zeta}(1-\tanh t)^{\eta} {\mathsf{Q}}_{k_{1},m}^{n}\left(\varsigma_{1}, \varsigma_{2}\right)\nonumber \\
&  = \frac{1}{2^{k_{1}/2}}\left(  1-\tanh^{2}x_{1}\right)^{\alpha}(1+\tanh t)^{\zeta+\frac{k_{1}}{2}}(1-\tanh t)^{\eta}\\
 &   \times P_{m-k_{1}}^{(k_{1}+\beta+\mu,\gamma)}\left( -\tanh t \right)
\ C_{k_{1}}^{\left(  \mu \right)  }\left(  \tanh x_{1}\right),\nonumber
\label{1-dim}%
\end{align}
where $\varsigma_{1}=\frac{1+\tanh t}{2}$ and $\varsigma_{2}=\left( \frac{1+\tanh t}{2}\right)^{1/2} \tanh x_{1}.$ The Fourier transform corresponding to this function takes the form
\begin{multline}
\mathcal{F}
\left( h_{m,k_{1}}\left( t, x_{1};\alpha,\zeta,\eta,\beta,\gamma,\mu \right)  \right)    =\int
\limits_{-\infty}^{\infty}\int
\limits_{-\infty}^{\infty}e^{-i\xi_{1}x_{1}-i\xi_{2}t}\left(  1-\tanh^{2}x_{1}\right)^{\alpha}(1+\tanh t)^{\zeta+\frac{k_{1}}{2}} \\
   \times \frac{1}{2^{k_{1}/2}}(1-\tanh t)^{\eta}P_{m-k_{1}}^{(k_{1}+\beta+\mu,\gamma)}\left( -\tanh t \right)
\ C_{k_{1}}^{\left(  \mu \right)  }\left(  \tanh x_{1}\right)dx_{1}dt \\
   =2^{\zeta+\eta+2\alpha-2}\frac{ \left( k_{1} + \mu + \beta + 1 \right)_{m-k_{1}} \Gamma \left( \zeta +\frac{k_{1}}{2} - \frac{i\xi_{2}}{2} \right)\Gamma \left( \eta+ \frac{i\xi_{2}}{2} \right)(2\mu)_{k_{1}}}{\left( m-k_{1} \right)!{k_{1}}!\Gamma \left(\frac{k_{1}}{2}+\zeta+\eta\right)}   \\
  \times \varphi_{1}^{1}\left(  \alpha,\mu,k_{1};\xi_{1}\right) \Theta \left(  m,k_{1},\zeta,\eta,\beta,\gamma,\mu,\xi_{2}\right)
\end{multline}
where%
\[
\varphi_{1}^{1}\left( \alpha,\mu,k_{1};\xi_{1}\right)  =B\left( \alpha+\frac{i\xi_{1}}{2},\alpha-\frac{i\xi_{1}}{2}\right)  \hyper{3}{2}{-k_{1},\ k_{1}+2\mu,\ \alpha+\frac{i\xi_{1}}{2}}{\mu+\frac{1}{2},\ 2\alpha}{1},
\]
and 
\[
\Theta \left(  m,k_{1},\zeta,\eta,\beta,\gamma,\mu,\xi_{2}\right) = \hyper{3}{2}{-m+k_{1},\ m+\mu+\beta+\gamma+1,\ \frac{k_{1}}{2}+\zeta-\frac{i\xi_{2}}{2}}{k_{1}+\mu+\beta+1,\ \frac{k_{1}}{2}+\zeta+\eta}{1},
\]
equivalently, this may be expressed in terms of the continuous Hahn polynomials
$p_{m}\left(  x;a,b,c,d\right)  $ as
\begin{multline*}%
\mathcal{F}\left(  h_{m,k_{1}}\left( t, x_{1};\alpha,\zeta,\eta,\beta,\gamma,\mu \right) \right)  \\
= 2^{\zeta+\eta+2\alpha-2} \frac{ \left( k_{1} + \mu + \beta + 1 \right)_{m-k_{1}} \Gamma \left( \zeta +\frac{k_{1}}{2} - \frac{i\xi_{2}}{2} \right) \Gamma \left(\eta+ \frac{i\xi_{2}}{2} \right) (2\mu)_{k_{1}}}{\Gamma \left( \frac{k_{1}}{2} + \zeta + \eta \right) i^{m} (2\alpha)_{k_{1}} \left( \mu + \frac{1}{2} \right)_{k_{1}}}\\
\times \frac{B \left( \alpha + \frac{i\xi_{1}}{2}, \alpha- \frac{i\xi_{1}}{2} \right)}{\left( k_{1} + \mu + \beta + 1 \right)_{m-k_{1}} \left( \frac{k_{1}}{2} + \zeta+ \eta \right)_{m-k_{1}} }\\
 \times p_{m-k_{1}}\left( \frac{-\xi_{2}}{2}; \frac{k_{1}}{2} + \zeta, \gamma-\eta+ 1, \frac{k_{1}}{2} + \mu + \beta-\zeta+1, \eta\right)\\
 \times p_{k_{1}}\left( \frac{\xi_{1}}{2}; \alpha, \mu -\alpha + \frac{1}{2}, \mu-\alpha + \frac{1}{2}, \alpha \right).
\end{multline*}

\end{proof}

\begin{proof}[Proof of Theorem \ref{theorem:34}]
The proof follows by induction on  $d$. To initiate the process, we consider the case  $d=1$, which yields the particular function (\ref{jac-spec-func})
\begin{align*}
  h_{m,k_{1}}\left( t, x_{1};\alpha_1,\zeta_1,\eta_1,\beta_1,\gamma_1,\mu_1 \right)    &  =\left(  1-\tanh^{2}x_{1}\right)^{\alpha_1}(1+\tanh t)^{\zeta_1}(1-\tanh t)^{\eta_1} {\mathsf{Q}}_{k_{1},m}^{n}\left(\varsigma_{1}, \varsigma_{2}\right) \\
&  = \frac{1}{2^{k_{1}/2}}\left(  1-\tanh^{2}x_{1}\right)^{\alpha_1}(1+\tanh t)^{\zeta_1+\frac{k_{1}}{2}}(1-\tanh t)^{\eta_1}\\
 &   \times P_{m-k_{1}}^{(k_{1}+\beta_1+\mu_1,\gamma_1)}\left( -\tanh t \right)
\ C_{k_{1}}^{\left(  \mu_1 \right)  }\left(  \tanh x_{1}\right),
\end{align*}
where $\varsigma_{1}=\frac{1+\tanh t}{2}$ and $\varsigma_{2}=\left( \frac{1+\tanh t}{2}\right)^{1/2} \tanh x_{1}.$
Inserting the function
$h_{m,k_{1}}$ and its Fourier transform into Parseval's identity%
\begin{align*}
&  4\pi^{2}%
{\displaystyle \int \limits_{-\infty}^{\infty}}
{\displaystyle \int \limits_{-\infty}^{\infty}}
h_{m,k_{1}}\left(  t,x_{1};\alpha_{1},\zeta_{1},\eta_{1},\beta_{1},\gamma
_{1},\mu_{1}\right)  \overline{h_{m',k_{1}'}\left(  t,x_{1};\alpha_{2},\zeta_{2},\eta
_{2},\beta_{2},\gamma_{2},\mu_{2}\right)}  dx_{1}dt\\
&  =\int \limits_{-\infty}^{\infty}\int \limits_{-\infty}^{\infty}
\mathcal{F}
\left(  h_{m,k_{1}}\left(  t,x_{1};\alpha_{1},\zeta_{1},\eta_{1},\beta
_{1},\gamma_{1},\mu_{1}\right)  \right)  \overline{\mathcal{F}
\left(  h_{m',k_{1}'}\left(  t,x_{1};\alpha_{2},\zeta_{2},\eta_{2},\beta
_{2},\gamma_{2},\mu_{2}\right)  \right) }d\xi_{1}d\xi_{2},
\end{align*}
leads%
\begin{align*}
&  4\pi^{2}%
{\displaystyle \int \limits_{-\infty}^{\infty}}
{\displaystyle \int \limits_{-\infty}^{\infty}}
\left(  1-\tanh^{2}x_{1}\right)  ^{\alpha_{1}+\alpha_{2}}\left(  1+\tanh
t\right)  ^{\zeta_{1}+\zeta_{2}}\left(  1-\tanh t\right)  ^{\eta_{1}+\eta_{2}%
}\mathsf{Q}_{k_{1},m}^{n}\left(  \varsigma
_{1},\varsigma_{2}\right)  \mathsf{Q}_{k_{1}',m'}^{n'%
}\left(  \varsigma_{1},\varsigma_{2}\right)  dx_{1}dt\\
&  =4\pi^{2}%
{\displaystyle \int \limits_{-\infty}^{\infty}}
{\displaystyle \int \limits_{-\infty}^{\infty}}
\frac{1}{2^{\frac{k_{1}+k_{1}'}{2}}}\left(  1-\tanh^{2}x_{1}\right)
^{\alpha_{1}+\alpha_{2}}\left(  1+\tanh t\right)  ^{\zeta_{1}+\zeta_{2}%
+\frac{k_{1}+k_{1}'}{2}}\left(  1-\tanh t\right)  ^{\eta_{1}+\eta_{2}}\\
&  \times P_{m-k_{1}}^{\left(  k_{1}+\mu_{1}+\beta_{1},\gamma_{1}\right)
}\left(  -\tanh t\right)  P_{m'-k_{1}'}^{\left(  k_{1}'+\mu_{2}+\beta_{2}%
,\gamma_{2}\right)  }\left(  -\tanh t\right)  C_{k_{1}}^{\left(  \mu
_{1}\right)  }\left(  \tanh x_{1}\right)  C_{k_{1}'}^{\left(  \mu_{2}\right)
}\left(  \tanh x_{1}\right)  dx_{1}dt\\
&  =\pi^{2}2^{\zeta_{1}+\zeta_{2}+\eta_{1}+\eta_{2}+1}%
{\displaystyle \int \limits_{0}^{1}}
u^{\zeta_{1}+\zeta_{2}+\frac{k_{1}+k_{1}'}{2}-1}\left(  1-u\right)  ^{\eta
_{1}+\eta_{2}-1}P_{m-k_{1}}^{\left(  k_{1}+\mu_{1}+\beta_{1},\gamma
_{1}\right)  }\left(  1-2u\right)  P_{m'-k_{1}'}^{\left(  k_{1}'+\mu_{2}%
+\beta_{2},\gamma_{2}\right)  }\left(  1-2u\right)  du\\
&  \times%
{\displaystyle \int \limits_{-1}^{1}}
\left(  1-t^{2}\right)  ^{\alpha_{1}+\alpha_{2}-1}C_{k_{1}}^{\left(  \mu
_{1}\right)  }\left(  t\right)  C_{k_{1}'}^{\left(  \mu_{2}\right)  }\left(
t\right)  dt\\
&  =\frac{2^{2\alpha_{1}+2\alpha_{2}+\zeta_{1}+\zeta_{2}+\eta_{1}+\eta_{2}%
-4}\left(  2\mu_{1}\right)  _{k_{1}}\left(  2\mu_{2}\right)  _{k_{1}'}\left(
k_{1}+\mu_{1}+\beta_{1}+1\right)  _{m-k_{1}}\left(  k_{1}'+\mu_{2}+\beta
_{2}+1\right)  _{m'-k_{1}'}}{\left(  m-k_{1}\right)  !\left(  m'-k_{1}'\right)
!k_{1}!k_{1}'!\Gamma \left(  2\alpha_{1}\right)  \Gamma \left(  2\alpha
_{2}\right)  \Gamma \left(  \frac{k_{1}}{2}+\zeta_{1}+\eta_{1}\right)
\Gamma \left(  \frac{k_{1}'}{2}+\zeta_{2}+\eta_{2}\right)  }\\
&  \times%
{\displaystyle \int \limits_{-\infty}^{\infty}}
{\displaystyle \int \limits_{-\infty}^{\infty}}
\text{ }\Gamma \left(  \alpha_{1}+\frac{i\xi_{1}}{2}\right)  \Gamma \left(
\alpha_{1}-\frac{i\xi_{1}}{2}\right)  \Gamma \left(  \eta_{1}+\frac{i\xi_{2}%
}{2}\right)  \Gamma \left(  \zeta_{1}+\frac{k_{1}}{2}-\frac{i\xi_{2}}%
{2}\right)  \\
&  \times \overline{\Gamma \left(  \alpha_{2}+\frac{i\xi_{1}}{2}\right)
\Gamma \left(  \alpha_{2}-\frac{i\xi_{1}}{2}\right)  }\overline{\Gamma \left(
\eta_{2}+\frac{i\xi_{2}}{2}\right)  \Gamma \left(  \zeta_{2}+\frac{k_{1}'}%
{2}-\frac{i\xi_{2}}{2}\right)  }\\
&  \times \hyper{3}{2}{-k_{1},\ k_{1}+2\mu_{1},\  \alpha_{1}+\frac{i\xi_{1}}%
{2}}{2\alpha_{1},\  \mu_{1}+1/2}{1}\ \overline{\hyper{3}{2}{-k_{1}',\ k_{1}'+2\mu_{2},\  \alpha_{2}+\frac{i\xi_{1}}
{2}}{2\alpha_{2},\  \mu_{2}+1/2}{1}}\\
&  \times \hyper{3}{2}{-m+k_{1},\ m+\mu_{1}+\beta_{1}+\gamma_{1}+1,\ \frac{k_{1}%
}{2}+\zeta_{1}-\frac{i\xi_{2}}{2}}{k_{1}+\mu_{1}+\beta_{1}+1,\ \frac{k_{1}}
{2}+\zeta_{1}+\eta_{1}}{1}\\
&  \times \overline{\hyper{3}{2}{-m'+k_{1}',\ m'+\mu_{2}+\beta_{2}+\gamma_{2}+1,\ \frac{k_{1}'%
}{2}+\zeta_{2}-\frac{i\xi_{2}}{2}}{k_{1}'+\mu_{2}+\beta_{2}+1,\ \frac{k_{1}'}
{2}+\zeta_{2}+\eta_{2}}{1}}d\xi_{1}d\xi_{2}.
\end{align*}
By assuming
\begin{align*}
\mu_{1} &  =\mu_{2}=\alpha_{1}+\alpha_{2}-\frac{1}{2}\\
\beta_{1} &  =\beta_{2}=\zeta_{1}+\zeta_{2}-\alpha_{1}-\alpha_{2}-\frac{1}%
{2}\\
\gamma_{1} &  =\gamma_{2}=\eta_{1}+\eta_{2}-1
\end{align*}
and considering the orthogonality relations of (\ref{ort}) and (\ref{ort-J}) ,
it is seen that the special function%
\begin{align*}
\ A_{m,k_{1}}^{\left(  2\right)  }\left(  t,x_{1};\alpha_{1},\alpha_{2}%
,\zeta_{1},\zeta_{2},\eta_{1},\eta_{2}\right)   &  =\ D_{k_1}\left( x_1;\alpha_{1},\alpha_{2}\right) 
\Gamma \left(  \zeta_{1}+\frac{k_{1}}{2}-\frac{t}{2}\right)  \\
&  \times\hyper{3}{2}{-m+k_{1},\ m+\zeta_{1}+\zeta_{2}+\eta_{1}+\eta
_{2}-1,\ \frac{k_{1}}{2}+\zeta_{1}-\frac{t}{2}}{k_{1}+\zeta_{1}+\zeta
_{2},\ \frac{k_{1}}{2}+\zeta_{1}+\eta_{1}}{1}
\end{align*}
where 
\begin{align*}
D_{k_1}\left( x_1;\alpha_{1},\alpha_{2}\right)=\Gamma \left(  \alpha
_{1}-\frac{x_{1}}{2}\right)  \Gamma \left(  \alpha_{1}+\frac{x_{1}}{2}\right)\hyper{3}{2}{-k_{1},\ k_{1}+2\left(  \alpha_{1}+\alpha_{2}\right)
-1,\  \alpha_{1}+\frac{x_{1}}{2}}{\alpha_{1}+\alpha_{2},\ 2\alpha_{1}}{1},
\end{align*}
alternatively in terms of continuous Hahn polynomials
\begin{align*}
\ A_{m,k_{1}}^{\left(  2\right)  }\left(  t,x_{1};\alpha_{1},\alpha_{2}%
,\zeta_{1},\zeta_{2},\eta_{1},\eta_{2}\right)   &  =\frac{\left(
m-k_{1}\right)  !k_{1}!i^{-m}}{\left(  2\alpha_{1}\right)  _{k_{1}}\left(
\alpha_{1}+\alpha_{2}\right)  _{k_{1}}\left(  k_{1}+\zeta_{1}+\zeta
_{2}\right)  _{m-k_{1}}\left(  \frac{k_{1}}{2}+\zeta_{1}+\eta_{1}\right)
_{m-k_{1}}}\\
&  \times~\Gamma \left(  \alpha_{1}-\frac{x_{1}}{2}\right)  \Gamma \left(
\alpha_{1}+\frac{x_{1}}{2}\right)  \Gamma \left(  \zeta_{1}+\frac{k_{1}}%
{2}-\frac{t}{2}\right)  \\
&  \times p_{k_{1}}\left(  \frac{-ix_{1}}{2};\alpha_{1},\alpha_{2},\alpha
_{2},\alpha_{1}\right)  p_{m-k_{1}}\left(  \frac{it}{2};\zeta_{1}+\frac{k_{1}%
}{2},\eta_{2},\zeta_{2}+\frac{k_{1}}{2},\eta_{1}\right)
\end{align*}
satisfies the relation%
\begin{align*}
&  \int \limits_{-\infty}^{\infty}\int \limits_{-\infty}^{\infty}\Gamma \left(
\eta_{1}+\frac{it}{2}\right)  \Gamma \left(  \eta_{2}-\frac{it}{2}\right)
A_{m,k_{1}}^{\left(  2\right)  }\left(  it,ix_{1};\alpha_{1},\alpha_{2}%
,\zeta_{1},\zeta_{2},\eta_{1},\eta_{2}\right)  \\
&  \times \ A_{m',k_{1}'}^{\left(  2\right)  }\left(  -it,-ix_{1};\alpha
_{2},\alpha_{1},\zeta_{2},\zeta_{1},\eta_{2},\eta_{1}\right)  dx_{1}dt\\
&  =\frac{\pi^{2}2^{-2\alpha_{1}-2\alpha_{2}+5}h_{k_{1}}^{\left(  \alpha
_{1}+\alpha_{2}-\frac{1}{2}\right)  }\Gamma \left(  m+\zeta_{1}+\zeta_{2}\right)
\Gamma \left(  m-k_{1}+\eta_{1}+\eta_{2}\right)  }{\left(  \left(  k_{1}+\zeta
_{1}+\zeta_{2}\right)  _{m-k_{1}}\right)  ^{2}\left(  \left(  2\alpha
_{1}+2\alpha_{2}-1\right)  _{k_{1}}\right)  ^{2}}\\
&  \times \frac{\left(  k_{1}!\right)  ^{2}\left(  m-k_{1}\right)
!\Gamma \left(  \frac{k_{1}}{2}+\zeta_{1}+\eta_{1}\right)  \Gamma \left(
\frac{k_{1}}{2}+\zeta_{2}+\eta_{2}\right)  \Gamma \left(  2\alpha_{1}\right)
\Gamma \left(  2\alpha_{2}\right)  }{\left(  2m-k_{1}+\zeta_{1}+\zeta_{2}%
+\eta_{1}+\eta_{2}-1\right)  \Gamma \left(  m+\zeta_{1}+\zeta_{2}+\eta_{1}%
+\eta_{2}-1\right)  }\delta_{k_{1},k_{1}'}\delta_{m,m'}%
\end{align*}
where the expression for $h_{k_{1}}^{\left(  \alpha_{1}+\alpha_{2}-\frac{1}%
{2}\right)  }$ is provided in (\ref{gnorm}). Analogously, inserting
(\ref{jac-spec-func}) and (\ref{Jac-Fourier}) into Parseval's identity
(\ref{multi}), leads, after some computations, to the expected conclusion.
\end{proof}
\begin{proof}[Proof of Theorem \ref{theorem:rec1}]
We first recall some well-known contiguous relations for the hypergeometric function $_{3}F_{2}$ that might be obtained by considering the Zeilberger’s algorithm \cite{Z1} based on the works \cite{F1,R2} as follows:
\begin{align}
 \beta\hyper{3}{2}{\alpha,\ \beta+1,\ \gamma}{\delta,\ \varepsilon}{z}
-\alpha\hyper{3}{2}{\alpha+1,\ \beta,\ \gamma}{\delta,\ \varepsilon}{z}  =\left(  \beta-\alpha \right)  \hyper{3}{2}{\alpha,\ \beta,\ \gamma}{\delta,\ \varepsilon}{z},  \label{rec2}
\end{align}
\[
\delta\hyper{3}{2}{\alpha,\ \beta,\ \gamma}{\delta,\ \varepsilon+1}{z}-\varepsilon\hyper{3}{2}{\alpha,\ \beta,\ \gamma}{\delta+1,\ \varepsilon}{z} = \left(  \delta-\varepsilon \right)  \hyper{3}{2}{\alpha,\ \beta,\ \gamma}{\delta+1,\ \varepsilon+1}{z},
\]
\[
\varepsilon \left(
\delta-\alpha \right)  \hyper{3}{2}{\alpha,\ \beta,\ \gamma}{\delta+1,\ \varepsilon}{z} -\delta \left(  \varepsilon-\alpha \right)  \hyper{3}{2}{\alpha,\ \beta,\ \gamma}{\delta,\ \varepsilon+1}{z} =\alpha \left(  \delta-\varepsilon \right)  \hyper{3}{2}{\alpha+1,\ \beta,\ \gamma}{\delta+1,\ \varepsilon+1}{z},
\]
\[
\gamma \left(  \beta-\alpha \right) z \hyper{3}{2}{\alpha+1,\ \beta+1,\ \gamma+1}{\delta+1,\ \varepsilon+1}{z}   =\delta \varepsilon\hyper{3}{2}{\alpha,\ \beta+1,\ \gamma}{\delta,\ \varepsilon}{z} \\
 -\delta \varepsilon\hyper{3}{2}{\alpha+1,\ \beta,\ \gamma}{\delta,\ \varepsilon}{z},
\]
\[
\alpha \left(  \delta-\beta \right)
 \hyper{3}{2}{\alpha+1,\ \beta,\ \gamma}{\delta+1,\ \varepsilon}{z}  -\beta \left(  \delta-\alpha \right)   \hyper{3}{2}{\alpha,\ \beta+1,\ \gamma}{\delta+1,\ \varepsilon}{z}=\delta \left(  \alpha-\beta \right)  \hyper{3}{2}{\alpha,\ \beta,\ \gamma}{\delta,\ \varepsilon}{z},
 \]
\[
\delta\hyper{3}{2}{\alpha,\ \beta,\ \gamma}{\delta,\ \varepsilon}{z}
+\left(  \alpha-\delta \right)  \hyper{3}{2}{\alpha,\ \beta,\ \gamma}{\delta+1,\ \varepsilon}{z} =\alpha\hyper{3}{2}{\alpha+1,\ \beta,\ \gamma}{\delta+1,\ \varepsilon}{z}
\]
and%
\[
\frac{\alpha \beta}{\delta \varepsilon}\hyper{3}{2}{\alpha+1,\ \beta+1,\ \gamma+1}{\delta+1,\ \varepsilon+1}{z} =\hyper{3}{2}{\alpha,\ \beta,\ \gamma+1}{\delta,\ \varepsilon}{z}  -\hyper{3}{2}{\alpha,\ \beta,\ \gamma}{\delta,\ \varepsilon}{z}  .
\]
If we get $\alpha \rightarrow-m+\left \vert \mathbf{k}\right \vert ,~\beta
\rightarrow m+\left \vert \zeta \right \vert +\left \vert \eta \right \vert
-1,~\gamma \rightarrow \frac{\left \vert \mathbf{k}\right \vert }{2}+\zeta
_{1}-\frac{t}{2},~\delta \rightarrow \left \vert \mathbf{k}\right \vert
+\left \vert \zeta \right \vert ,~\varepsilon \rightarrow \frac{\left \vert
\mathbf{k}\right \vert }{2}+\zeta_{1}+\eta_{1}$ and $z\rightarrow1$ in
(\ref{rec2}), and we use the definition of the function $A_{m,\mathbf{k}%
}^{\left(  d+1\right)  }\left(  t,\boldsymbol{x};\alpha_{1},\alpha_{2},\zeta
_{1},\zeta_{2},\eta_{1},\eta_{2}\right)  $ we obtain the relation in $(i)$.
Similarly, applying the contiguous relations given above respectively gives
the relations in $(ii)-(vii)$.
\end{proof}

\begin{proof}[Proof of Theorem \ref{theorem:31}]
The Lemma \ref{prop:OPcone2} allows us to compute the Fourier transform of the function $h_{m,\mathbf{k}}\left( t,\boldsymbol{x};\alpha,\zeta,\beta,\mu \right)$ specified in (\ref{100}). It is derived from (\ref{100})
\begin{multline}
\mathcal{F}\left(  h_{m,\mathbf{k}}\left( t,\boldsymbol{x};\alpha,\zeta,\beta,\mu \right) \right)   ={\displaystyle \int \limits_{-\infty}^{\infty}}
{\displaystyle \int \limits_{-\infty}^{\infty}}\cdots {\displaystyle \int \limits_{-\infty}^{\infty}}
 h_{m,\mathbf{k}}\left( t,\boldsymbol{x};\alpha,\zeta,\beta,\mu \right) e^{-i\left(  \xi_{1}x_{1}+\cdots+\xi_{d}x_{d}+\xi_{d+1}t\right)  }d\boldsymbol{x}dt \nonumber \\
    ={\displaystyle \int \limits_{0}^{\infty}}
{\displaystyle \int \limits_{-\infty}^{\infty}}\cdots {\displaystyle \int \limits_{-\infty}^{\infty}}
e^{-i\left(  \xi_{1}x_{1}+\cdots+\xi_{d}x_{d}\right)}g_{d}\left(
x_{1},\dots,x_{d};k_{1},\dots,k_{d},\alpha,\mu \right) \nonumber  \\
   \times e^{-u/2} L_{m-\left \vert \mathbf{k}\right \vert }^{\left \vert \mathbf{k}\right \vert+\mu+\beta+\frac{d-1}{2}}\left(u\right) u^{\zeta+\frac{\left \vert \mathbf{k}\right \vert}{2} -i\xi_{d+1}-1}d\boldsymbol{x}du \nonumber\\
   =\mathcal{F}\left(g_{d}\left(
x_{1},\dots,x_{d};k_{1},\dots,k_{d},\alpha,\mu \right)\right)\nonumber\\
   \times {\displaystyle \int \limits_{0}^{\infty}}e^{-u/2} L_{m-\left \vert \mathbf{k}\right \vert }^{\left \vert \mathbf{k}\right \vert+\mu+\beta+\frac{d-1}{2}}\left(u\right) u^{\zeta+\frac{\left \vert \mathbf{k}\right \vert}{2} -i\xi_{d+1}-1}du\nonumber.
\end{multline}
If we use the definition of Laguerre polynomials (\ref{102}) and Gamma function, we arrive at
\begin{multline}
\mathcal{F}\left(  h_{m,\mathbf{k}}\left( t,\boldsymbol{x};\alpha,\zeta,\beta,\mu \right) \right)=\mathcal{F}\left(g_{d}\left(
x_{1},\dots,x_{d};k_{1},\dots,k_{d},\alpha,\mu \right)\right) \frac{\left(  \left \vert \mathbf{k}\right \vert +\mu+\beta+\frac{d+1}{2}\right)
_{m-\left \vert \mathbf{k}\right \vert }}{\left(  m-\left \vert \mathbf{k}%
\right \vert \right)  !} \nonumber \\
   \times 2^{\zeta+\frac{\left \vert \mathbf{k}\right \vert}{2} -i\xi_{d+1}%
}\Gamma \left(  \zeta+\frac{\left \vert \mathbf{k}\right \vert}{2} -i\xi_{d+1}\right)
\hyper{2}{1}{-m+\left \vert \mathbf{k}\right \vert ,\ \zeta+\frac{\left \vert
\mathbf{k}\right \vert}{2} -i\xi_{d+1}}{\left \vert \mathbf{k}\right \vert
+\mu+\beta+\frac{d+1}{2}}{2}. \nonumber
\end{multline}
Using \eqref{18}, the proof is completed.

We now turn our attention to a particular case of the general result. When $d=1$, the special function takes the form
\begin{align}
h_{m,k_{1}}\left( t, x_{1};\alpha,\zeta,\beta,\mu \right)    &  =\left(  1-\tanh^{2}x_{1}\right)^{\alpha} {\mathsf{R}}_{k_{1},m}^{n}\left(\sigma_{1}, \sigma_{2}\right) e^{-\frac{e^{t}}{2}+\zeta t} \nonumber \\
   &  =\left(  1-\tanh^{2}x_{1}\right)^{\alpha} L_{m-k_{1}}^{k_{1}+\mu+\beta}\left( e^{t} \right) e^{-\frac{e^{t}}{2}+\zeta t+\frac{k_{1}t}{2}}
\ C_{k_{1}}^{\left(  \mu \right)  }\left(  \tanh x_{1}\right),\nonumber
\label{g1dim}%
\end{align}
where $\sigma_{1}=e^{t}$ and $\sigma_{2}=e^{t/2}\tanh x_{1},$ and its Fourier transform is
\begin{multline}
\mathcal{F}
\left( h_{m,k_{1}}\left( t, x_{1};\alpha,\zeta,\beta,\mu \right)  \right)   \\
=\int \limits_{-\infty}^{\infty}\int \limits_{-\infty}^{\infty}e^{-i\xi_{1}x_{1}-i\xi_{2}t}\left(  1-\tanh^{2}x_{1}\right)^{\alpha} L_{m-k_{1}}^{k_{1}+\mu+\beta}\left( e^{t} \right) e^{-\frac{e^{t}}{2}+\zeta t+\frac{k_{1}t}{2}}
\ C_{k_{1}}^{\left(  \mu \right)  }\left(  \tanh x_{1}\right)dx_{1}dt \nonumber \\
  =\frac{2^{2\alpha+\zeta+\frac{k_{1}}{2}-i\xi_{2}-1}\left(  k_{1}+\mu+\beta+1\right)  _{m-k_{1}}~\left(
2\mu \right)  _{k_{1}}}{k_{1}!\left(  m-k_{1}\right)  !} \\ \times \Gamma \left(  \zeta+\frac{k_{1}}{2}-i\xi_{2}\right)
   \Lambda \left(  m,k_{1},\zeta,\mu,\beta ,\xi_{2}\right) \varphi_{1}^{1}\left(  \alpha,\mu,k_{1};\xi_{1}\right)
\end{multline}
where%
\[
\varphi_{1}^{1}\left(  \alpha,\mu,k_{1};\xi_{1}\right)  =B\left(  \alpha+\frac{i\xi_{1}}{2},\alpha-\frac{i\xi_{1}}{2}\right) \hyper{3}{2}{-k_{1},\ k_{1}+2\mu,\ \alpha+\frac{i\xi_{1}}{2}}{\mu+\frac{1}
{2},\ 2\alpha}{1},
\]
and
\[
\Lambda \left(  m,k_{1},\zeta,\mu,\beta ,\xi_{2}\right) =\hyper{2}{1}{-m+k_{1},\ \zeta+\frac{k_{1}}{2}-i\xi_{2}}{k_{1}+\mu+\beta+1}{2}.
\]
One may rewrite it using the continuous Hahn polynomials $p_{m}\left(  x;a,b,c,d\right)$ defined by \eqref{hahn} in the following form
\begin{align*}
\mathcal{F}
\left( h_{m,k_{1}}\left( t, x_{1};\alpha,\zeta,\beta,\mu \right)  \right)   &  =\frac{2^{2\alpha+\zeta+\frac{k_{1}}{2}-i\xi_{2}-1}\left(  k_{1}+\mu+\beta+1\right) _{m-k_{1}} \left(2\mu \right)_{k_{1}} \Gamma \left( \zeta+\frac{k_{1}}{2}-i\xi_{2}\right)}{i^{k_{1}}\left(  m-k_{1}\right)! \left(  2\alpha\right) _{k_{1}}\left(  \mu+1/2 \right)_{k_{1}}}\\
&  \times B\left(\alpha+\frac{i\xi_{1}}{2},\ \alpha-\frac
{i\xi_{1}}{2}\right) \Lambda \left(  m,k_{1},\zeta,\mu,\beta ,\xi_{2}\right)\\
&  \times p_{k_{1}}\left(  \frac{\xi_{1}}{2};\alpha,\mu-\alpha+1/2,\mu-\alpha+1/2,\alpha \right).
\end{align*}
\end{proof}

\begin{proof}[Proof of Theorem \ref{theorem:32}]
The proof follows by induction on  $d$. To initiate the process, we consider the case  $d=1$, which yields the particular function (\ref{100})
\begin{align*}
h_{m,k_{1}}\left( t, x_{1};\alpha,\zeta,\beta,\mu \right) =\left(  1-\tanh^{2}x_{1}\right)^{\alpha} L_{m-k_{1}}^{k_{1}+\mu+\beta}\left( e^{t} \right) e^{-\frac{e^{t}}{2}+\zeta t+\frac{k_{1}t}{2}}
\ C_{k_{1}}^{\left(  \mu \right)  }\left(  \tanh x_{1}\right).
\end{align*}
Using Parseval's identity, one derives

\begin{align*}
&  4\pi^{2}%
{\displaystyle \int \limits_{-\infty}^{\infty}}
{\displaystyle \int \limits_{-\infty}^{\infty}}
\left(  1-\tanh^{2}x_{1}\right)  ^{\alpha_{1}+\alpha_{2}}\exp \left(
-e^{t}+\left(  \zeta_{1}+\zeta_{2}\right)  t+\frac{\left(  k_{1}+k_{1}'\right)
t}{2}\right)  \\
&  \times L_{m-k_{1}}^{k_{1}+\mu_{1}+\beta_{1}}\left(  e^{t}\right)
L_{m'-k_{1}'}^{k_{1}'+\mu_{2}+\beta_{2}}\left(  e^{t}\right)  C_{k_{1}}^{\left(
\mu_{1}\right)  }\left(  \tanh x_{1}\right)  C_{k_{1}'}^{\left(  \mu
_{2}\right)  }\left(  \tanh x_{1}\right)  dx_{1}dt\\
&  =4\pi^{2}%
{\displaystyle \int \limits_{0}^{\infty}}
L_{m-k_{1}}^{k_{1}+\mu_{1}+\beta_{1}}\left(  u\right)  L_{m'-k_{1}'}^{k_{1}'%
+\mu_{2}+\beta_{2}}\left(  u\right)  e^{-u}u^{\zeta_{1}+\zeta_{2}+\frac
{k_{1}+k_{1}'}{2}-1}du\\
&  \times%
{\displaystyle \int \limits_{-1}^{1}}
\left(  1-v^{2}\right)  ^{\alpha_{1}+\alpha_{2}-1}C_{k_{1}}^{\left(  \mu
_{1}\right)  }\left(  v\right)  C_{k_{1}'}^{\left(  \mu_{2}\right)  }\left(
v\right)  dv\\
&  =\frac{2^{2\left(  \alpha_{1}+\alpha_{2}-1\right)  +\zeta_{1}+\zeta
_{2}+\frac{k_{1}+k_{1}'}{2}}\left(  2\mu_{1}\right)  _{k_{1}}\left(  2\mu
_{2}\right)  _{k_{1}'}\left(  k_{1}+\mu_{1}+\beta_{1}+1\right)  _{m-k_{1}%
}\left(  k_{1}'+\mu_{2}+\beta_{2}+1\right)  _{m'-k_{1}'}}{\left(  m-k_{1}\right)
!\left(  m'-k_{1}'\right)  !k_{1}!k_{1}'!\Gamma \left(  2\alpha_{1}\right)
\Gamma \left(  2\alpha_{2}\right)  }\\
&  \times%
{\displaystyle \int \limits_{-\infty}^{\infty}}
{\displaystyle \int \limits_{-\infty}^{\infty}}
\text{ }\Gamma \left(  \alpha_{1}+\frac{i\xi_{1}}{2}\right)  \Gamma \left(
\alpha_{1}-\frac{i\xi_{1}}{2}\right)  ~\overline{\Gamma \left(  \alpha
_{2}+\frac{i\xi_{1}}{2}\right)  \Gamma \left(  \alpha_{2}-\frac{i\xi_{1}}%
{2}\right)  }\\
&  \times \Gamma \left(  \zeta_{1}+\frac{k_{1}}{2}-i\xi_{2}\right)
~\overline{\Gamma \left(  \zeta_{2}+\frac{k_{1}'}{2}-i\xi_{2}\right)  }\\
&  \times\hyper{3}{2}{-k_{1},\ k_{1}+2\mu_{1},\  \alpha_{1}+\frac{i\xi_{1}}
{2}}{2\alpha_{1},\  \mu_{1}+1/2}{1}\ \overline{\hyper{3}{2}{-k_{1}',\ k_{1}'+2\mu_{2},\  \alpha_{2}+\frac{i\xi_{1}}
{2}}{2\alpha_{2},\  \mu_{2}+1/2}{1}}\\
&  \times\hyper{2}{1}{-m+k_{1},\  \zeta_{1}+\frac{k_{1}}{2}-i\xi_{2}}{k_{1}
+\mu_{1}+\beta_{1}+1}{2} \ \overline{\hyper{2}{1}{-m'+k_{1}',\  \zeta_{2}+\frac{k_{1}'}{2}-i\xi_{2}}{k_{1}'
+\mu_{2}+\beta_{2}+1}{2}}d\xi_{1}d\xi_{2}.
\end{align*}
If we choose
\begin{align*}
\mu_{1} &  =\mu_{2}=\alpha_{1}+\alpha_{2}-\frac{1}{2}\\
\beta_{1} &  =\beta_{2}=\zeta_{1}+\zeta_{2}-\alpha_{1}-\alpha_{2}-\frac{1}{2}%
\end{align*}
together with the orthogonality conditions of (\ref{ort}) and (\ref{ort-L}),
one observes that the special function%
\[
\ B_{m,k_{1}}^{\left(  2\right)  }\left(  x_{1},t;\alpha_{1},\alpha_{2}%
,\zeta_{1},\zeta_{2}\right)  =\  \Gamma \left(  \zeta_{1}+\frac{k_{1}}%
{2}-t\right)  D_{k_{1}}\left(  x_{1};\alpha_{1},\alpha_{2}\right) \hyper{2}{1}{-m+k_{1},\  \zeta_{1}+\frac{k_{1}}{2}-t}{k_{1}+\zeta
_{1}+\zeta_{2}}{2}
\]
where%
\[
D_{k_{1}}\left(  x_{1};\alpha_{1},\alpha_{2}\right)  =\Gamma \left(  \alpha
_{1}-\frac{x_{1}}{2}\right)  \Gamma \left(  \alpha_{1}+\frac{x_{1}}{2}\right)
\hyper{3}{2}{-k_{1},\ k_{1}+2\left(  \alpha_{1}+\alpha_{2}\right)
-1,\  \alpha_{1}+\frac{x_{1}}{2}}{\alpha_{1}+\alpha_{2},\ 2\alpha_{1}}{1},
\]
alternatively in terms of continuous Hahn polynomials
\begin{align*}
\ B_{m,k_{1}}^{\left(  2\right)  }\left(  x_{1},t;\alpha_{1},\alpha_{2}%
,\zeta_{1},\zeta_{2}\right)    & =\frac{k_{1}!i^{-k_{1}}}{\left(  2\alpha
_{1}\right)  _{k_{1}}\left(  \alpha_{1}+\alpha_{2}\right)  _{k_{1}}}%
\Gamma \left(  \alpha_{1}-\frac{x_{1}}{2}\right)  \Gamma \left(  \alpha
_{1}+\frac{x_{1}}{2}\right)  \Gamma \left(  \zeta_{1}+\frac{k_{1}}{2}-t\right)
\\
& \times\hyper{2}{1}{-m+k_{1},\  \zeta_{1}+\frac{k_{1}}{2}-t}{k_{1}+\zeta
_{1}+\zeta_{2}}{2}\ p_{k_{1}}\left(  \frac{-ix_{1}}{2};\alpha_{1}%
,\alpha_{2},\alpha_{2},\alpha_{1}\right)  ,
\end{align*}
satisfies the relation%
\begin{align*}
&  \int \limits_{-\infty}^{\infty}\int \limits_{-\infty}^{\infty}\  \ B_{m,k_{1}%
}^{\left(  2\right)  }\left(  it,ix_{1};\alpha_{1},\alpha_{2},\zeta_{1}%
,\zeta_{2}\right)  \ B_{m',k_{1}'}^{\left(  2\right)  }\left( -it, -ix_{1};\alpha_{1},\alpha_{2},\zeta_{1},\zeta_{2}\right)  dx_{1}dt\\
&  =\frac{\pi^{2}2^{-\left(  2\alpha_{1}+2\alpha_{2}+\zeta_{1}+\zeta_{2}%
+k_{1}\right)  +4}h_{k_{1}}^{\left(  \alpha_{1}+\alpha_{2}-\frac{1}{2}\right)
}\left(  k_{1}!\right)  ^{2}\left(  m-k_{1}\right)  !}{\left(  k_{1}+\zeta
_{1}+\zeta_{2}\right)  _{m-k_{1}}^{2}\left(  2\alpha_{1}+2\alpha_{2}-1\right)
_{k_{1}}^{2}}\\
&  \times \Gamma \left(  \zeta_{1}+\zeta_{2}+m\right)  \Gamma \left(  2\alpha
_{1}\right)  \Gamma \left(  2\alpha_{2}\right)  \delta_{k_{1},k_{1}'}%
\delta_{m,m'}%
\end{align*}
where $h_{k_{1}}^{\left(  \alpha_{1}+\alpha_{2}-\frac{1}{2}\right)  }$ is
given by (\ref{gnorm}). In the same way, an iterative application of
Parseval's identity with (\ref{15}) and (\ref{18}) yields the result.
\end{proof}

\begin{proof}[Proof of Theorem \ref{theorem:rec2}]
We first recall the well-known contiguous relations for the hypergeometric
function $_{2}F_{1}$ as follows \cite{25}:%
\begin{align}
\left(  1-z\right) \hyper{2}{1}{\alpha,\ \beta}{\gamma}{z} &
=\hyper{2}{1}{\alpha-1,\ \beta}{\gamma}{z}  -\gamma^{-1}\left(
\gamma-\beta \right)  z\hyper{2}{1}{\alpha,\ \beta}{\gamma+1}{z}
,\label{rec1}\\
\left(  1-z\right)  \hyper{2}{1}{\alpha,\ \beta}{\gamma}{z}  &
=\hyper{2}{1}{\alpha,\ \beta-1}{\gamma}{z}  -\gamma^{-1}\left(
\gamma-\alpha \right)  z\ \hyper{2}{1}{\alpha,\ \beta}{\gamma+1}{z}
,\nonumber \\
\left(  \alpha-\beta \right)  \hyper{2}{1}{\alpha,\ \beta}{\gamma}{z}
& =\alpha\hyper{2}{1}{\alpha+1,\ \beta}{\gamma}{z}  -\beta\hyper{2}{1}{\alpha,\ \beta+1}{\gamma}{z},\nonumber \\
\left(  \alpha-\gamma+1\right) \hyper{2}{1}{\alpha,\ \beta}{\gamma}{z}  & =\alpha\hyper{2}{1}{\alpha+1,\ \beta}{\gamma}{z}
-\left(  \gamma-1\right)  \hyper{2}{1}{\alpha,\ \beta}{\gamma-1}{z}
,\nonumber \\
\left(  \beta-\gamma+1\right)  \hyper{2}{1}{\alpha,\ \beta}{\gamma}{z}
& =\beta\hyper{2}{1}{\alpha,\ \beta+1}{\gamma}{z}  -\left(
\gamma-1\right)  \hyper{2}{1}{\alpha,\ \beta}{\gamma-1}{z}
,\nonumber \\
\left(  2\alpha-\gamma+z\left(  \beta-\alpha \right)  \right)  \hyper{2}{1}{\alpha,\ \beta}{\gamma}{z}  & =\alpha \left(  1-z\right)
\hyper{2}{1}{\alpha+1,\ \beta}{\gamma}{z} \nonumber \\
&\quad \quad \quad \quad \quad \quad \quad+\left(  \alpha-\gamma \right) \hyper{2}{1}{\alpha-1,\ \beta}{\gamma}{z}  ,\nonumber
\end{align}
and%
\begin{align*}
\left(  \left(  \gamma-\beta \right)  z-\alpha \right) \hyper{2}{1}{\alpha+1,\ \beta}{\gamma+1}{z} & =\gamma \left(  z-1\right)  \hyper{2}{1}{\alpha+1,\ \beta}{\gamma}{z} \\
& \quad \quad \quad \quad \quad\quad+\left(  \gamma-\alpha \right)  \hyper{2}{1}{\alpha,\ \beta}{\gamma+1}{z}.
\end{align*}
If we get $\alpha \rightarrow-m+\left \vert \mathbf{k}\right \vert ,~\beta
\rightarrow \zeta_{1}+\frac{\left \vert \mathbf{k}\right \vert }{2}%
-t,~\gamma \rightarrow \left \vert \mathbf{k}\right \vert +\left \vert
\zeta \right \vert $ and $z\rightarrow2$ in (\ref{rec1}), and we use the
definition of the function $B_{m,\mathbf{k}}^{\left(  d+1\right)  }\left(
t,\boldsymbol{x};\alpha_{1},\alpha_{2},\zeta_{1},\zeta_{2}\right)  $ we obtain the
relation in $(i)$. Similarly, applying the contiguous relations given above
respectively gives the relations in $(ii)-(vii)$.
\end{proof}

\textbf{Acknowledgements}
Not applicable.\\

\textbf{Authors' contributions}
Both authors contributed equally to this work. Both authors have read and approved the final manuscript.\\

\textbf{Funding} No funding.\\

\textbf{Data availability}
Data sharing is not applicable to this article as no data sets were generated or analyzed during the current study.\\

\section*{Declarations}
\textbf{Conflict of interest} The authors declare no competing interests.\\

\textbf{Ethical Approval} Not applicable.

\end{document}